\documentclass[12pt]{article}

\usepackage{mathptmx} 
\DeclareSymbolFont{cmletters}{OML}{cmm}{m}{it}              
\DeclareSymbolFont{cmsymbols}{OMS}{cmsy}{m}{n}
\DeclareSymbolFont{cmlargesymbols}{OMX}{cmex}{m}{n}
\DeclareMathSymbol{\myjmath}{\mathord}{cmletters}{"7C}     \let\jmath\myjmath %Defining the missing commands: \jmath, \amalg and \coprod
\DeclareMathSymbol{\myamalg}{\mathbin}{cmsymbols}{"71}     \let\amalg\myamalg
\DeclareMathSymbol{\mycoprod}{\mathop}{cmlargesymbols}{"60}\let\coprod\mycoprod
\DeclareMathSymbol{\myalpha}{\mathord}{cmletters}{"0B}     \let\alpha\myalpha %Greek letters from Computer Modern
\DeclareMathSymbol{\mybeta}{\mathord}{cmletters}{"0C}      \let\beta\mybeta
\DeclareMathSymbol{\mygamma}{\mathord}{cmletters}{"0D}     \let\gamma\mygamma
\DeclareMathSymbol{\mydelta}{\mathord}{cmletters}{"0E}     \let\delta\mydelta
\DeclareMathSymbol{\myepsilon}{\mathord}{cmletters}{"0F}   \let\epsilon\myepsilon
\DeclareMathSymbol{\myzeta}{\mathord}{cmletters}{"10}      \let\zeta\myzeta
\DeclareMathSymbol{\myeta}{\mathord}{cmletters}{"11}       \let\eta\myeta
\DeclareMathSymbol{\mytheta}{\mathord}{cmletters}{"12}     \let\theta\mytheta
\DeclareMathSymbol{\myiota}{\mathord}{cmletters}{"13}      \let\iota\myiota
\DeclareMathSymbol{\mykappa}{\mathord}{cmletters}{"14}     \let\kappa\mykappa
\DeclareMathSymbol{\mylambda}{\mathord}{cmletters}{"15}    \let\lambda\mylambda
\DeclareMathSymbol{\mymu}{\mathord}{cmletters}{"16}        \let\mu\mymu
\DeclareMathSymbol{\mynu}{\mathord}{cmletters}{"17}        \let\nu\mynu
\DeclareMathSymbol{\myxi}{\mathord}{cmletters}{"18}        \let\xi\myxi
\DeclareMathSymbol{\mypi}{\mathord}{cmletters}{"19}        \let\pi\mypi
\DeclareMathSymbol{\myrho}{\mathord}{cmletters}{"1A}       \let\rho\myrho
\DeclareMathSymbol{\mysigma}{\mathord}{cmletters}{"1B}     \let\sigma\mysigma
\DeclareMathSymbol{\mytau}{\mathord}{cmletters}{"1C}       \let\tau\mytau
\DeclareMathSymbol{\myupsilon}{\mathord}{cmletters}{"1D}   \let\upsilon\myupsilon
\DeclareMathSymbol{\myphi}{\mathord}{cmletters}{"1E}       \let\phi\myphi
\DeclareMathSymbol{\mychi}{\mathord}{cmletters}{"1F}       \let\chi\mychi
\DeclareMathSymbol{\mypsi}{\mathord}{cmletters}{"20}       \let\psi\mypsi
\DeclareMathSymbol{\myomega}{\mathord}{cmletters}{"21}     \let\omega\myomega
\DeclareMathSymbol{\myvarepsilon}{\mathord}{cmletters}{"22}\let\varepsilon\myvarepsilon
\DeclareMathSymbol{\myvartheta}{\mathord}{cmletters}{"23}  \let\vartheta\myvartheta
\DeclareMathSymbol{\myvarpi}{\mathord}{cmletters}{"24}     \let\varpi\myvarpi
\DeclareMathSymbol{\myvarrho}{\mathord}{cmletters}{"25}    \let\varrho\myvarrho
\DeclareMathSymbol{\myvarsigma}{\mathord}{cmletters}{"26}  \let\varsigma\myvarsigma
\DeclareMathSymbol{\myvarphi}{\mathord}{cmletters}{"27}    \let\varphi\myvarphi

\usepackage{amsthm,amsmath,amssymb,amscd,graphics,enumerate,latexsym,verbatim,stmaryrd,graphicx}
\usepackage[all]{xy}%\CompileMatrices\CompilePrefix{xymatrix/diagram}

\theoremstyle{plain}
\newtheorem{thm}{Theorem}[section]
\newtheorem{cor}[thm]{Corollary}
\newtheorem{lemma}[thm]{Lemma}
\newtheorem{prop}[thm]{Proposition}

\theoremstyle{definition}
\newtheorem{df}[thm]{Definition}

\newtheorem{pg}[thm]{}
\newtheorem{rem}[thm]{Remark}
\newtheorem{ex}[thm]{Example}

\DeclareMathOperator{\SL}{SL}
\DeclareMathOperator{\Quot}{Quot}

\DeclareMathOperator{\Spec}{Spec}
\DeclareMathOperator{\CSpec}{CSpec}
\DeclareMathOperator{\Sch}{Sch}
\DeclareMathOperator{\BSch}{BSch}
\DeclareMathOperator{\res}{res}
\DeclareMathOperator{\Hom}{Hom}

\DeclareMathOperator{\Rad}{Rad}
\DeclareMathOperator{\colim}{colim}

\DeclareMathOperator{\im}{im}

\DeclareMathOperator{\eq}{eq}

\DeclareMathOperator{\zero}{\textrm{\o}}

\def\0{{\bf 0}}

\def\F{{\mathbb F}}

\def\N{{\mathbb N}}
\def\P{{\mathbb P}}
\def\Q{{\mathbb Q}}
\def\R{{\mathbb R}}
\def\Z{{\mathbb Z}}

\def\cB{{\mathcal B}}
\def\cC{{\mathcal C}}
\def\cD{{\mathcal D}}
\def\cE{{\mathcal E}}
\def\cF{{\mathcal F}}
\def\cG{{\mathcal G}}
\def\cI{{\mathcal I}}

\def\cM{{\mathcal M}}

\def\cO{{\mathcal O}}

\def\cR{{\mathcal R}}

\def\cU{{\mathcal U}}

\def\fm{{\mathfrak m}}
\def\fp{{\mathfrak p}}
\def\fq{{\mathfrak q}}

\def\fB{{\mathfrak B}}

\def\Fun{{\F_1}}
\def\Funsq{{\F_{1^2}}}
\def\Funn{{\F_{1^n}}}

\def\int{\textup{int}}

\def\id{\textup{id}}

\def\blanc{-}
\def\bp{{\mathcal{B}lpr}}
\def\bpspaces{{\mathcal{L}\!oc\bp \mathcal{S}\!p}}
\def\canc{{\textup{canc}}}
\def\inv{{\textup{inv}}}
\def\proper{{\textup{prop}}}
\def\gl{{\textup{gl}}}
\def\SRings{{\mathcal{SR}ings}}
\def\Rings{{\mathcal{R}ings}}
\def\={\equiv}
\def\n={\equiv\hspace{-10,5pt}/\hspace{3,5pt}}

\newdir{ >}{{}*!/-5pt/@^{(}}\newcommand{\arincl}[1]{\ar@{ >->}@<-0,0ex>#1} %inclusion arrow for xy-matrix with better spacing

\newcommand{\norm}[1]{\left| #1 \right|}
\newcommand{\gen}[1]{\langle #1 \rangle}
\newcommand{\igen}[1]{\langle #1 \rangle}
\newcommand{\bpquot}[2]{#1\!\sslash\!#2}
\newcommand{\bpgenquot}[2]{#1\!\sslash\!\gen{#2}}

\begin{document}

\title{\bf The geometry of blueprints\\[0,8cm] \large\bf Part I: Algebraic background and scheme theory}
\author{Oliver Lorscheid\footnote{Department of Mathematics, The City College of New York, 160 Convent Ave., New York NY 10031, USA., phone +1-212-650-5120, olorscheid@ccny.cuny.edu.}}
\date{}
%\address{Department of Mathematics, The City College of New York, 160 Convent Ave., New York NY 10031, USA.}
%\email{olorscheid@ccny.cuny.edu.}

\maketitle
%{\ \vspace{-250pt}\\ \flushright\tiny\bf Version 0.5\\ \today\\ }\vspace{218pt}

\begin{abstract}
 In this paper, we introduce the category of blueprints, which is a category of algebraic objects that include both commutative (semi)rings and commutative monoids. This generalization allows a simultaneous treatment of ideals resp.\ congruences for rings and monoids and leads to a common scheme theory. In particular, it bridges the gap between usual schemes and $\Fun$-schemes (after Kato \cite{Kato94}, Deitmar \cite{Deitmar05} and Connes-Consani \cite{CC09}). Beside this unification, the category of blueprints contains new interesting objects as ``improved'' cyclotomic field extensions $\Funn$ of $\Fun$ and ``archimedean valuation rings''. It also yields a notion of semiring schemes.

 This first paper lays the foundation for subsequent projects, which are devoted to the following problems: Tits' idea of Chevalley groups over $\Fun$, congruence schemes, sheaf cohomology, $K$-theory and a unified view on analytic geometry over $\Fun$, adic spaces (after Huber), analytic spaces (after Berkovich) and tropical geometry.\\[5pt]
\end{abstract}

\newpage
\begin{footnotesize}\tableofcontents\end{footnotesize}

%%%%%%%%%%%%%%%%%%%%%%%%%%%%%%%%%%%%%%%%%%%%%%%%%%%%%%%%%%%%%%%%%%%%%%%%%%%%%%%%%%%%%%%%%%%%%%%%%%%%%%%%%%%%%%%%%%%%%%%%%%%%%%%%%%%%%%%%%%%%%%%%%%%%%%%%%%%%%%%%%%%%%%%%%%%%%%%%%%%%%%%%%%%%%%%%%%%%%%%%%%%%%%%%%%%%%%%%%%%%
%%%%%%%%%%%%%%%%%%%%%%%%%%%%%%%%%%%%%%%%%%%%%%%%%%%%%%%%%%%%%%%%%%%%%%%%%%%%%%%%%%%%%%%%%%%%%%%%%%%%%%%%%%%%%%%%%%%%%%%%%%%%%%%%%%%%%%%%%%%%%%%%%%%%%%%%%%%%%%%%%%%%%%%%%%%%%%%%%%%%%%%%%%%%%%%%%%%%%%%%%%%%%%%%%%%%%%%%%%%%

%\newpage
\section*{Introduction}
\label{intro}\addcontentsline{toc}{section}{Introduction}

In the last decade, a series of suggestions were made what an $\Fun$-geometry should be. These approaches generalized scheme theory from different viewpoints: Soul\'e's approach via endowing a complex algebra with descent data to $\Fun$ (\cite{Soule04}), Durov's approach via monads (\cite{Durov07}), To\"en-Vaqui\'e's approach via functors on the category of multiplicative monoids (\cite{ToenVaquie08}), Borger's approach via endowing schemes with a $\Lambda$-structure (\cite{Borger09}), Haran's approach via categories of coherent sheaves (\cite{Haran07}, \cite{Haran09}), Kato, Deitmar and Connes-Consani's approach via a scheme theory associated to monoids (\cite{Kato94},\cite{Deitmar05},\cite{CC09}), or the approach of Connes-Consani resp.\ of L\'opez Pe\~na and the author via decomposition of schemes into tori (\cite{CC09}, \cite{LL09}). For a more comprehensive overview, see \cite{LL09b}.

What can be learned from all these generalizations of scheme theory is that the important operation for applying methods from algebraic geometry is the multiplicative structure while addition plays a subordinated role. Most of the named geometries are either too restrictive to cover an interesting class of objects or too artificial resp.\ too abstract to be applicable to standard methods from algebraic geometry. We follow a different approach in this text: a blueprint is a generalization of a commutative ring that is general enough to contain monoids and semirings, but close enough to conventional algebraic structures to develop a scheme theory in the vein of algebraic geometry after Grothendieck. To contrast the theory of blueprints to other generalizations of scheme theory, we avoid the name $\Fun$-scheme in this text.

While the author's initial interest was an investigation of Tits' idea of Chevalley groups over $\Fun$ (cf.\ \cite{Tits56}; it has been realized in \cite{L09}), it turned out that the theory of blueprints has multiple connections to other theories and might be suitable for a deeper understanding of other geometric concepts. We first outline the definition of a blueprint and a blue scheme, before we explain these ideas.

\subsubsection*{Summary of the theory}

For simplicity, we will restrict ourselves to proper blueprint with $0$, which we briefly call blueprints in this exposition. See Lemma \ref{lemma:semirings_and_blueprints} for the connection with the axiomatic approach in the main text.

A \emph{blueprint} is a pair $B=(A,R)$ of a semiring $R$ and a multiplicative subset $A$ of $R$ that contains $0$ and $1$ and that generates $R$ as a semiring. A \emph{morphism $f:B_1\to B_2$ of blueprints} is a semiring morphism $f_R:R_1\to R_2$ that restricts to a multiplicative map $f_A:A_1\to A_2$. An \emph{ideal} of a blueprint $B=(A,R)$ is a subset $I$ of $A$ such that there is a morphism $f:B\to B'$ with $f_A^{-1}(0)=I$. A \emph{prime ideal} is an ideal $\fp$ of $B$ such that $A-\fp$ is a multiplicative subset of $A$ that contains $1$. The localization of $B=(A,R)$ at a multiplicative subset $S\subset A$ that contains $1$ is the blueprint $S^{-1}B=(S^{-1}A,S^{-1}R)$ (cf.\ \cite{CC09} for the definition of $S^{-1}A$). If $S=A-\fp$ for a prime ideal $\fp$ of $B$, then we denote $S^{-1}B$ by $B_\fp$.

These properties make it possible to define the \emph{spectrum of a blueprint $B=(A,R)$} as the set $\Spec B$ of all prime ideals endowed with the topology generated by open subsets of the form $U_h=\{\fp\in\Spec B|h\notin \fp\}$ where $h$ ranges through $A$ together with a structure sheaf that is associated to the stalks $B_\fp$ at $\fp\in\Spec B$. A \emph{blue scheme} is a topological space $X$ together with a sheaf $\cO_X$ of blueprints that is locally isomorphic to spectra of blueprints. The usual definition of morphisms applies: morphisms of blue schemes are local morphisms of ``locally blueprinted spaces''.

While most of the theory generalizes scheme theory in an obvious way, a discrepancy arises in the fact that the global sections $\Gamma B=\Gamma(X,\cO_X)$ of $X=\Spec B$ are in general not equal to $B$. However, there is a canonical blueprint morphism $B\to \Gamma B$, which induces an isomorphism $\Spec\Gamma B\to \Spec B$ of blue schemes (Theorem \ref{thm:global_sections_are_global_blueprints}). This allows to prove properties of blue schemes analogous to usual schemes, e.g.\ that morphisms are locally algebraic.

We can consider a semiring $R$ as a blueprint $B=(R,R)$. This defines a fully faithful embedding of the category of semirings into the category of blueprints. In case of rings, a prime ideal of $R$ is the same as a prime ideal of $B$, and the spectrum of $R$ is the same as the spectrum of $B$. This leads to a fully faithful embedding of usual schemes into blue schemes. Additionally, we obtain the notion of a semiring scheme. Similarly, a monoid $A$ with a zero $0$ defines a blueprint $B=(A,A_\Z)$ where $A_\Z=\Z[A]/(0)$ and $(0)$ is the ideal generated by $0\in A\subset\Z[A]$. This leads to a fully faithful embedding of $\Fun$-schemes (after Kato, Deitmar and Connes-Consani, cf.\ \cite{CC09}) into blue schemes. Also the recent theory of sesquiads and their Zariski schemes (\cite{Deitmar11}) embeds fully faithfully into the theory of blue schemes.

\subsubsection*{Future prospects}

The paper lays the foundation for a number of projects, which we will describe in the following. Roughly speaking, these concern two main streams of ideas: making sense of Jacques Tits' dream of explaining thin geometries by an algebraic geometry over $\Fun$ (cf.\ \cite{Tits56}) and developing a unified view on algebraic geometry, $\Fun$-geometry, tropical geometry and analytic geometry (in the sense of Berkovich and Huber). 

Note that we forgo intentionally to give the term $\Fun$-geometry a precise meaning. It can be understood, in a more strict sense, to mean the class of monoidal schemes (after Kato, \cite{Kato94}, Deitmar, \cite{Deitmar05}, and Connes-Consani, \cite{CC09}), or, in a more lax way, to mean all kinds of geometric objects of a combinatorial nature. A posteriori, we can consider blue schemes over $\Fun$ as an $\Fun$-geometry. This is a broad class of objects, which includes next to monoidal schemes examples as the $\Fun$-model $\Spec\bigl(\bpgenquot{\Fun[X,Y]}{X^2+Y^2\=1}\bigr)$ of the unit circle or $\Fun$-models of moduli spaces like, for instance, the projective line 
\[
 \P^1_{\Fun}-\{0,1,\infty\} \quad = \quad \Spec\bigl(\ \bpgenquot{\Fun[X^{\pm1},Y^{\pm1}]}{X\=Y+1}\ \bigr)
\]
over $\Fun$ with $0$, $1$ and $\infty$ omitted. Examples of a more arithmetic nature are ``archimedean valuation rings'' (see Section \ref{section:archimedean_valuation_rings}) or the ``improved'' cyclotomic field extensions $\Funn$ (see Section \ref{section:cyclotomic_field_extensions}) whose base extension to $\Z$ is contained in a cyclotomic field extension of $\Q$. These properties are important for a number theory of blueprints that incorporates a Frobenious action in ``indefinite characteristics'' and allows us to consider archimedean and non-archimedean places simultaneously.

Most of the papers on $\Fun$-geometry address Chevalley groups over $\Fun$ in one or the other way (e.g.\ \cite{CC08,CC09,Deitmar05,Deitmar11,KapranovSmirnov,LL09,L09,L10,Manin95,Soule04,ToenVaquie08}). It became clear that the Weyl group $W$ has an important meaning for the postulated geometry of a Chevalley group $G$ over $\Fun$. In particular, there are many ways, in which one can give meaning to the formula $G(\Fun)=W$. But all explanations so far are of a rather formal nature and withstand an explanation of thin geometries in terms of an algebraic geometry over $\Fun$. In particular, while the role of the Weyl group is understood in a sufficient way, ``the structure of the terms of higher order is [still] more mysterious'' (\cite[p.\ 25]{CC08}).

It was indeed this problem that initially led to the definition of a blueprint: there was a need to give a meaning to expressions like $\Fun[X_1,\dotsc,X_4]/(X_1X_4=X_2X_3+1)$, e.a., which should play the role of the global sections of a Chevalley group over $\Fun$  (in this case $\SL_2$). With the theory of blue schemes, it is possible to explain the connection between the algebraic geometry of Chevalley groups and the combinatorial geometry of their Weyl groups. Moreover, the role of the additive structure will become clear, and it is possible to extend the group law of a Chevalley group to all blueprints, which, for instance, includes semirings. This might help to clarify the role of algebraic groups in tropical geometry along the connection that we mention below. All this will be explained in a subsequent paper, which is based on the methods developed in the present work. 

The latter project is of a more extensive nature and involves several aspects. First of all, it seems to be necessary to establish congruence schemes in the generality of blueprints. There are treatments of congruence schemes in a more restricted setting by Berkovich (for monoids, cf.\ \cite{Berkovich11}) and Deitmar (for sesquiads, cf.\ \cite{Deitmar11}). The important class of congruence spectra of semirings is contained in neither approach. Before one can extend definitions (which differ in the named approaches) to the more general setting of blueprints, one has to answer certain technical questions, which arise around the definition of the structure sheaf. In particular, the analogue of Theorem \ref{thm:global_sections_are_global_blueprints} is desirable. In this paper, we provide all algebraic background on congruences, but we postpone to give the definition of a congruence scheme till the theory is settled enough to be formulated.

With these tools at hands, a unified view on algebraic geometry, $\Fun$-geometry, analytic geometry and tropical geometry seems possible. We establish already in this paper a common approach to usual algebraic geometry, $\Fun$-geometry and an algebraic geometry associated to semirings. Congruence schemes will add to this unified approach, and will be of particular importance when we encounter semiring schemes. For example, tropical schemes as defined by Mikhalkin (cf.\ \cite{Mikhalkin}) seem to have a natural explanation as congruence schemes of tropical semirings.

Paugam introduced in \cite{Paugam09} a category whose objects are ordered semirings such that valuations of fields resp.\ rings are morphisms in this category. This allows us to recover Berkovich's and Huber's analytic spaces from homomorphism sets of Paugam's category. This category fits very naturally in the setting of (ordered) blueprints. Thus we can adopt Paugam's idea and extend it to the geometric setting that is developed in this paper. As a byproduct, this yields the notion of an analytic space over $\Fun$. Of particular interest is to compare this viewpoint with the theory of $\Fun$-analytic spaces of Berkovich (cf.\ \cite{Berkovich11}), which is based on congruence schemes. 

The explicit nature of blue schemes as topological spaces with structure sheaves makes it possible to extend sheaf theory of usual schemes and of monoidal schemes (as developed in \cite{CLS10}) to the more general setting of blue schemes. This enables us to give a unified approach towards $K$-theory for usual schemes and for $\Fun$-schemes. We expect that $K$-theory finds its most naturally explanation within the framework of congruence schemes since the particular role that idempotent elements play for congruence schemes promises to dissolve the gap between locally projective and locally free sheaves that occurs in sheaf theory of $\Fun$-schemes (cf.\ \cite{CLS10}). As a side result, this will define $K$-theory for tropical schemes and analytic spaces. Furthermore, the definition of \v{C}ech cohomology extends to blue schemes over $\Funsq$, which yields a definition for sheaf cohomology.

\subsubsection*{Content overview}

We outline the structure of this paper. In Section \ref{section:blueprints}, we introduce the notion of a blueprint and define certain full subcategories, including the subcategories of monoids (with and without zero), semirings and rings. We review the algebra behind $\Fun$-geometry after Deitmar (\cite{Deitmar05}, \cite{Deitmar11}) and Connes-Consani (\cite{CC09}) from the perspective of blueprints. We define the cyclotomic field extensions $\Funn$ of $\Fun$ and discuss archimedean valuation rings. We conclude the section with certain constructions: free blueprints, localizations, small limits and finite resp. directed colimits.

In Section \ref{section:congruences_and_ideals}, we introduce the notions of a congruence and of an ideal of a blueprint. We give alternative characterizations of congruences and ideals for certain subclasses of blueprints: ideals in blueprints with a zero have a simpler description than in the general context; congruences and ideals for monoids coincide with the usual definitions; blueprint ideals for rings coincide with usual ideals, and every congruence arises from an ideal (which is the reason why there is no need for the congruences in ring theory); in the case of semirings, a blueprint ideal is what is called a $k$-ideal in semiring theory. We conclude the section by introducing prime and maximal congruences and ideals together with some basic facts.

In Section \ref{section:blue_schemes}, we define locally blueprinted spaces and the spectrum of a blueprint. A blue scheme is a locally blueprinted space that is locally isomorphic to spectra of blueprints. We introduce the globalization $B\to \Gamma B$ where $\Gamma B$ is the blueprint of global sections of $\Spec B$, which is in general not an isomorphism (in contrast to usual scheme theory). We prove that the induced morphism $\Spec\Gamma B\to\Spec B$ is an isomorphism of blue schemes. This allows us to prove that the affine open subschemes of a blue scheme form a basis for its topology and that every morphism of blue schemes is locally algebraic. We define residue fields of points of a blue scheme. By restricting the category of blueprints to those that come from monoids or ring, we recover the notions of $\cM$-schemes (resp.\ $\cM_0$-schemes if we consider monoids with a zero) and usual schemes. Furthermore, we obtain the notion of a semiring scheme. We conclude this section with a comparison with $\Fun$-schemes in the sense of Connes and Consani (\cite{CC09}).

\subsubsection*{Acknowledgements}

%\textbf{Acknowledgements:} 

I would like to thank all those who took the time to discuss my ideas on blueprints with me. This includes, but does not restrict to, the following persons: Vladimir Berkovich, Ethan Cotterill, Anton Deitmar, Paul Lescot, Javier L\'opez Pe\~na, Florian Pop, Matt Szczesny, Lucien Szpiro and Annette Werner. I would like to thank Gautam Chinta for receiving me as a research guest at CUNY and for his hospitality during this year.

%%%%%%%%%%%%%%%%%%%%%%%%%%%%%%%%%%%%%%%%%%%%%%%%%%%%%%%%%%%%%%%%%%%%%%%%%%%%%%%%%%%%%%%%%%%%%%%%%%%%%%%%%%%%%%%%%%%%%%%%%%%%%%%%%%%%%%%%%%%%%%%%%%%%%%%%%%%%%%%%%%%%%%%%%%%%%%%%%%%%%%%%%%%%%%%%%%%%%%%%%%%%%%%%%%%%%%%%%%%%
%%%%%%%%%%%%%%%%%%%%%%%%%%%%%%%%%%%%%%%%%%%%%%%%%%%%%%%%%%%%%%%%%%%%%%%%%%%%%%%%%%%%%%%%%%%%%%%%%%%%%%%%%%%%%%%%%%%%%%%%%%%%%%%%%%%%%%%%%%%%%%%%%%%%%%%%%%%%%%%%%%%%%%%%%%%%%%%%%%%%%%%%%%%%%%%%%%%%%%%%%%%%%%%%%%%%%%%%%%%%

\section{Blueprints} 
\label{section:blueprints}

%%%%%%%%%%%%%%%%%%%%%%%%%%%%%%%%%%%%%%%%%%%%%%%%%%%%%%%%%%%%%%%%%%%%%%%%%%%%%%%%%%%%%%%%%%%%%%%%%%%%%%%%%%%%%%%%%%%%%%%%%%%%%%%%%%%%%%%%%%%%%%%%%%%%%%%%%%%%%%%%%%%%%%%%%%%%%%%%%%%%%%%%%%%%%%%%%%%%%%%%%%%%%%%%%%%%%%%%%%%%

\subsection{Definition}
\label{section:blueprint_definition}

%In this first section, we introduce the notion of a blueprint. 

%\begin{pg}
 Throughout this text, a \emph{monoid} is a (multiplicatively written) commutative semi-group $A$ with a neutral element $1$. A \emph{morphism $f:A_1\to A_2$ of monoids} is a multiplicative map with $f(1)=1$. The semiring $\N[A]$ is defined as the (additive) semi-group of all finite formal sums $\sum a_i$, where the $a_i$ are elements in $A$, possibly occurring multiple times in the sum. The empty sum, denoted by $\zero$, is the neutral element for addition. The multiplication of $\N[A]$ is defined by linear extension of the multiplication of $A$, and $1$, considered as formal sum in $\N[A]$, is the neutral element for the multiplication of $\N[A]$. 

 In the following, we will consider relations $\cR \subset\N[A]\times\N[A]$ on $\N[A]$, which should be thought of equalities of sums of elements in $A$. Therefore, we write intuitively $\sum a_i\=\sum b_j$ for the condition that $(\sum a_i,\sum b_j)$ is an element of $\cR$.
%\end{pg}

\begin{df}
 A \emph{pre-addition for $A$} is a relation $\cR$ on $\N[A]$ that satisfies the following axioms for all $a_i,b_j,c_k,d_l\in A$.
 \begin{enumerate}
  \item[\textup{(A1)}] $\sum a_i\=\sum a_i$.                                                                          \textit{\hfill(reflexive)}     
  \item[\textup{(A2)}] If $\sum a_i\= \sum b_j$, then $\sum b_j\=\sum a_i$.                                           \textit{\hfill(symmetric)}     
  \item[\textup{(A3)}] If $\sum a_i\=\sum b_j$ and $\sum b_j\=\sum c_k$, then $\sum a_i\=\sum c_k$.                   \textit{\hfill(transitive)}    
  \item[\textup{(A4)}] If $\sum a_i\=\sum b_j$ and $\sum c_k\=\sum d_l$, then $\sum a_i+\sum c_k\=\sum b_j+\sum d_l$. \textit{\hfill(additive)}      
  \item[\textup{(A5)}] If $\sum a_i\=\sum b_j$ and $\sum c_k\=\sum d_l$, then $\sum a_i\cdot c_k\=\sum b_j\cdot d_l$. \textit{\hfill(multiplicative)}
  \end{enumerate}
 A pre-addition $\cR$ on $A$ may satisfy one of the following additional axioms. 
 \begin{enumerate}\addtocounter{enumi}{5}
  \item[\textup{(A6)}] If $\sum a_i+\sum c_k\=\sum b_j+\sum c_k$, then $\sum a_i\=\sum b_j$.                              \textit{\hfill(cancellative)}   
  \item[\textup{(A7)}] For every $a\in A$, there is a $b\in A$ such that $a+b\=\zero$.                                    \textit{\hfill(inverses)}      
  \item[\textup{(A8)}] There is an $a\in A$ such that $a\=\zero$.                                                          \textit{\hfill(zero)}          
  \item[\textup{(A9)}] If $a\= b$, then $a=b$ as elements of $A$.                                                        \textit{\hfill(proper)}        
 \end{enumerate}
 We call a pre-addition $\cR$ satisfying one of these additional axioms \emph{a cancellative pre-addition}, \emph{a pre-addition with inverses}, \emph{a pre-addition with a zero} resp.\ \emph{a proper pre-addition}.
\end{df}

Note that axioms \textup{(A1)}--\textup{(A3)} mean that a pre-addition is an equivalence relation on $\N[A]$, and that axioms \textup{(A4)} and \textup{(A5)} mean that a pre-addition is a sub-semiring of the product semiring $\N[A]\times\N[A]$.

\begin{df}
 A \emph{blueprint} is a pair $(A,\cR)$ that consists of a monoid $A$ and a pre-addition $\cR$ on $A$. A \emph{morphism $f:(A_1,\cR_1)\to(A_2,\cR_2)$ of blueprints} is a monoid morphism $\tilde f: A_1\to A_2$ such that the induced semiring morphism $\N[A_1]\times\N[A_1]\to\N[A_2]\times\N[A_2]$ maps $\cR_1$ to $\cR_2$. We denote the category of blueprints by $\bp$.

 A blueprint is \emph{cancellative}, \emph{with inverses}, \emph{with a zero} or \emph{proper} if $\cR$ is so.
\end{df}

If $A$ is a monoid and $\cR$ a pre-addition on $A$, we denote the blueprint $(A,\cR)$ by $B=\bpquot A\cR$. We often write $a\in B$ for $a\in A$ and $f(a)$ for $\tilde f(a)$. In this terminology, a morphism of blueprints is a multiplicative map $f: B_1\to B_2$ with $f(1)=1$ such that $\sum f(a_i)\=\sum f(b_j)$ for all $\sum a_i\=\sum b_j$. A third characterization is that a morphism $\bpquot {A_1}{\cR_1}\to\bpquot {A_2}{\cR_2}$ of blueprints is a semiring morphism $\cR_1\to\cR_2$ that maps $A_1\subset\cR_1$ (embedded as $a\= a$, by reflexivity) to $A_2\subset\cR_2$.

\begin{rem}
 Equivalently, one can define a blueprint to be a multiplicative map $f:A\to R$ from a monoid $A$ to a semiring $R$. The reader might find this viewpoint more accessible on a first reading and is encouraged to have a look at this alternative viewpoint in Section \ref{section:subsets_of_semirings}.
\end{rem}

%%%%%%%%%%%%%%%%%%%%%%%%%%%%%%%%%%%%%%%%%%%%%%%%%%%%%%%%%%%%%%%%%%%%%%%%%%%%%%%%%%%%%%%%%%%%%%%%%%%%%%%%%%%%%%%%%%%%%%%%%%%%%%%%%%%%%%%%%%%%%%%%%%%%%%%%%%%%%%%%%%%%%%%%%%%%%%%%%%%%%%%%%%%%%%%%%%%%%%%%%%%%%%%%%%%%%%%%%%%%

\subsection{First properties}
\label{section:first_properties}

To begin with, we list first properties concerning the additional axioms \textup{(A6)}--\textup{(A9)}. Beside some immediate implications from the axioms, we provide constructions that close a blueprint with respect to these additional axioms.

Motivated by the next lemma, we define the following subcategories of $\bp$. The category $\bp_\canc$ is the full subcategory of $\bp$ whose objects are \emph{cancellative blueprints}; $\bp_\proper$ is the full subcategory of $\bp$ whose objects are \emph{proper blueprints}; $\bp_\inv$ is the full subcategory of $\bp$ whose objects are \emph{proper blueprints with inverses}; $\bp_0$ is the full subcategory of $\bp$ whose objects are \emph{proper blueprints with a zero}.

\begin{lemma}\label{lemma_bp-fact}
 Let $ B$ be a blueprint.
 \begin{enumerate}
  \item\label{bp-fact1}\label{part1} If $B$ is with inverses, then $ B$ is cancellative.
  \item\label{bp-fact2}\label{part2} The blueprint $ B$ is with inverses if and only if $1$ has an inverse, i.e.\ if there is an element $a\in B$ such that $1+a\=\zero$.
  \item\label{bp-fact3}\label{part3} If $1+a\=\zero$, then $a^2\=1$.
  \item\label{bp-fact4}\label{part4} If $B$ has a zero $e\=\zero$, then $\sum b_j+e\=\sum b_j$ for all $b_j\in B$ and $e\cdot b\= e$ for all $b\in B$.
  \item\label{bp-fact5}\label{part5} If $B$ is proper, then $B$ has at most one zero. 
  \item\label{bp-fact6}\label{part6} If $B$ is proper, then inverses are unique and $(-1)^2=1$ if $-1$ is the (unique) inverse of $1$. 
  \item\label{bp-fact7}\label{part7} If $B$ has unique inverses, then $B$ is proper.
 \end{enumerate}
 If $a\in B$ has a unique inverse, we denote it by $-a$, and if $B$ has a unique zero, we denote it by $0$.
\end{lemma}

\begin{proof}
 Part \eqref{bp-fact1} follows immediately from the additivity of the pre-addition. The inclusion of \eqref{bp-fact2} is clear and the reverse inclusion follows by multiplying $1+a\=\zero$ by $b$, which yields the inverse $ab$ for $b$.

 Part \eqref{bp-fact3} is obtained as follows. If $1+a\=\zero$, then $a+a^2\=\zero$ after multiplying by $a$. Adding $1$ gives $1+a+a^2\=1$ and using $1+a\=\zero$ with additivity yields the claim $a^2\=1$.

 The first statement of part \eqref{bp-fact4} follows from $\sum b_j\=\sum b_j$ and $e\=\zero$ by additivity and the latter statement from $b\= b$ and $e\=\zero$ by multiplicativity of the pre-addition.

 Part \eqref{bp-fact5} can be deduced as follows. If $a$ and $b$ are zeros of $ B$, then by \eqref{bp-fact4}, $a\= ab\= b$, which are equalities in $ B$ by properness. 

 Part \eqref{bp-fact6} is obtained as follows. If $a+b\=\zero\= a+c$, then by additivity, $a+b+b\= a+c+b$ and using $a+b\=\zero$ on both sides, we have $b\= c$. By properness, we have thus $b=c$. The equation $(-1)^2=1$ follows from \eqref{bp-fact3} by properness. For the last statement \eqref{bp-fact7} assume that $a\= b$. The relation $a+(-a)\=\zero$ yields $b+(-a)\=\zero$ by additivity. The uniqueness of the inverse of $-a$ implies that $a=b$.
\end{proof}

\begin{pg}\label{pg:generated_pre-addition}
 Let $A$ be a monoid. Every relation $S$ on $\N[A]$ can be closed under the axioms of a pre-addition, i.e.\ there is a smallest pre-addition $\cR$ for $A$ that contains $S$. This is easily verified: the whole product $\N[A]\times\N [A]$ is a pre-addition and pre-additions are closed under intersections. We denote this smallest pre-addition by $\cR=\langle S\rangle$ and the generated blueprint by $ B(A,S)=\bpgenquot AS$. We say that $B(A,S)$ \emph{is generated by $S$ over $A$}, or, that \emph{$S$ is a set of generators for $\cR$}. By abuse of notation, we write, for instance, $\cR=\gen{a\= b, c\=d}$ or $\cR=\gen{a+a\=\zero}_{a\in A}$ for $\cR=\gen S$ where $S\subset\N[A]\times\N[A]$ equals $\{(a,b),(c,d)\}$ resp.\ $\{(a+a,\zero)|a\in A\}$.

 Let $f:A_1\to A_2$ be a morphism of monoids. To see whether a pre-addition $\cR_1$ for $A_1$ is mapped to a pre-addition $\cR_2$ for $A_2$, it is enough to verify whether a set of generators for $\cR_1$ is mapped to $\cR_2$.

 Similarly, there is a smallest pre-addition with axiom \textup{(A6)} containing $S$. We denote this cancellative pre-addition by $\gen S_\canc$ and the generated blueprint by $ B(A,S)_\canc=\bpquot A{\gen S_\canc}$. In particular, this defines $ B_\canc$ in $\bp_\canc$ for every blueprint $ B=\bpquot A\cR$ (take $S=\cR$). The natural inclusion $ B\to B_\canc$ is a morphism of blueprints that satisfies the universal property that every morphism from $ B$ into a cancellative blueprint factors uniquely through $ B_\canc$.

 The other axioms \textup{(A7)}--\textup{(A9)} behave somewhat differently. We begin with the properness axiom \textup{(A9)}. Note that we do not include properness in the definition of a blueprint, but list it as an additional axiom, in spite of the fact that all blueprints of interest will be proper. The reason for this is that we emphasize the multiplicative structure of a blueprint and want to think of the pre-addition as a structure imposed on monoids, and the pre-addition should not be a restriction for the multiplication. We will see in the following that we have to pass to a quotient of the underlying monoid in order to obtain a proper blueprint. 
\end{pg}

\begin{lemma} \label{lemma_proper}
 Let $B=\bpquot A\cR$ be a blueprint. Then there is a blueprint $B_\proper=(A_\proper,\cR_\proper)$ and a morphism $ B\to B_\proper$ such that for every blueprint $C$ in $\bp_\proper$ and every morphism $f: B\to C$, there is a unique morphism $f_\proper: B_\proper\to C$ such that the diagram
 $$\xymatrix{ B\ar[d]\ar[rr]^f &&  C \\  B_\proper\ar[urr]_{f_\proper}} $$
 commutes. If $B$ is cancellative, with inverses or with a zero, then $B_\proper$ is also cancellative, with inverses resp.\ with a zero.
\end{lemma}

\begin{proof}
 For the construction of $B_\proper$ we have to pass to a quotient of $A$ (see Section \ref{subsection:congruences} for more details on the construction of quotients). Let $\sim$ be the restriction of $\cR$ to $A\subset\N[A]$, i.e.\ the equivalence relation on $A$ that is defined by $a\sim b$ if and only if $a\= b$. Then $[a]\cdot[b]=[ab]$ is a well-defined multiplication for the quotient set $A/\sim$. Indeed, if $a\= a'$ and $b\= b'$, then $ab\= a'b'$; the class of $1\in A$ is the neutral element of $A/\sim$. Thus the quotient $A_\proper=A/\sim$ is a monoid. The relation $\cR_\proper$ defined by $\sum[a_i]\=\sum[b_j]$ if $\sum a_i\=\sum b_j$ in $\cR$ is a pre-addition for $A_\proper$: axioms \textup{(A1)}--\textup{(A5)} follow immediately from the corresponding axioms for $\cR$.

 The blueprint $ B_\proper=\bpquot {A_\proper}{\cR_\proper}$ is proper by construction. Given a morphism $f: B\to C$, we see that if $a\= b$ in $ B$, then $f(a)\= f(b)$ in $C$ and thus $f(a)=f(b)$ by properness of $C$. Therefore we can define $f_\proper$ on equivalence classes of $\sim$, which establishes the universal property of $B_\proper$.

 It is easily seen that $ B_\proper$ satisfies \textup{(A7)} resp.\ \textup{(A8)} if they are satisfied by $ B$. We present the slightly more involved case of the cancellation axiom \textup{(A6)}. If $\sum [a_i]+\sum [c_k]\=\sum [b_j]+\sum [c_k]$ in $B_\proper$, then $\sum a_i+\sum c_k\=\sum b_j+\sum c_k'$ for certain $c_k'\=c_k$ in $B$. By additivity, we have $\sum b_j+\sum c_k'\=\sum b_j+\sum c_k$ and by cancellation in $B$, it follows from $\sum a_i+\sum c_k\=\sum b_j+\sum c_k$ that $\sum a_i\=\sum b_j$ and thus $\sum [a_i]\=\sum [b_j]$, which proves cancellation for $B_\proper$.
\end{proof}

 In the cases of axioms \textup{(A7)} and \textup{(A8)}, we have to assume the uniqueness of inverses resp.\ the zero to make sense of a universal property as above. 

\begin{lemma} \label{lemma_inv}
 Let $ B$ be a blueprint. Then there exist the following universal objects.
 \begin{enumerate}
  \item For every blueprint $ B$, there is a morphism $ B\to B_\inv$ into a proper blueprint $ B_\inv$ with inverses such that every morphism from $ B$ into a proper blueprint with inverses factors uniquely through $ B_\inv$.
  \item For every blueprint $ B$, there is a morphism $ B\to B_0$ into a proper blueprint $ B_0$ with a zero such that every morphism from $ B$ into a proper blueprint with a zero factors uniquely through $ B_0$.
 \end{enumerate}
\end{lemma}

\begin{proof}
 We only prove the first part of the lemma. The second statement is proved in a similar way. 

 The construction of $ B_\inv$ will be in two steps. First we enrich $ B$ with inverses, and then we take the proper quotient, which has inverses by Lemma \ref{lemma_proper}. If $1$ has an inverse in $ B$, then $ B$ is with inverses by Lemma \ref{lemma_bp-fact} \eqref{bp-fact2} and we define $\tilde  B_\inv$ as $ B$. 

 If $1$ does not have an inverse, then we define $\tilde A_\inv$ as the union of $A$ with the set of symbols $-A=\{-a|a\in A\}$. We extend the multiplication of $A$ to $\tilde A_\inv$ by the rules $(-a)\cdot b=-(ab)=a\cdot (-b)$ and $(-a)\cdot(-b)=ab$. Further we define $\tilde\cR_\inv$ as the pre-addition on $\tilde A_\inv$ that is generated by $S=\cR\cup\{a+(-a)\=\zero|a\in A\}$. Define $\tilde  B_\inv$ as $\bpquot{\tilde A_\inv}{\tilde\cR_\inv}$, which is a blueprint with inverses by construction.

 As mentioned before, the blueprint $ B_\inv=(\tilde  B_\inv)_\proper$ is proper and with inverses by Lemma \ref{lemma_proper}. We show that it satisfies the universal property of the lemma. Let $f: B\to C$ be a morphism into a proper blueprint with inverses. First we extend $f$ to a morphism $\tilde f_\inv:\tilde  B_\inv\to C$. If $\tilde  B_\inv= B$, then there is nothing to show. If $\tilde A_\inv=A\cup -A$, then the relations $a+(-a)\=\zero$ force $\tilde f_\inv$ to map $-a$ to the (unique) inverse of $\tilde f_\inv(a)=f(a)$ in $C$. This defines the unique extension of $f$ to $\tilde  B_\inv$. The rest of the proof follows by Lemma \ref{lemma_proper}.
\end{proof}

Finally, we show that colimits exist in the subcategories $\bp_\cE$ if they exist in $\bp$ (where $\cE$ stays for one of the attributes $\canc$, $\inv$, $0$ or $\proper$) and that limits coincide while colimits are connected by the functor $(-)_\cE$ from $\bp$ to $\bp_\cE$. We make this precise in the following proposition. Note that the previous lemmas imply that the hypothesis of the proposition are satisfied for the inclusion $\iota:\bp_\cE\to\bp$ as subcategory and $\cG=(-)_\cE$.

\begin{prop}\label{prop:adjoint_functors}
 Let $\iota:\cC\to\cB$ be a functor with a left-adjoint $\cG:\cB\to\cC$ such that $\cG\circ\iota$ is isomorphic to the identity functor $\id_\cC$ on $\cC$. Let $\cD$ be a diagram in $\cC$.
 \begin{enumerate}
  \item If the limit $\lim\iota(\cD)$ exists in $\cB$ and lies in the essential image of $\iota$, then the limit $\lim\cD$ exists in $\cC$, and if the limit $\lim\cD$ exists, then $\lim\iota(\cD)\simeq\iota(\lim\cD)$. 
  \item If the colimit $\colim\iota(\cD)$ exists in $\cB$, then the colimit $\colim\cD$ exists in $\cC$ and $\colim\cD\simeq\cG(\colim\iota(\cD))$.
 \end{enumerate}
\end{prop}

\begin{proof}
 This proposition basically follows from the fact that left-adjoints preserve limits and right-adjoints preserve colimits. We execute the proof.

 Assume that $\lim\iota(\cD)$ is isomorphic to $\iota(B)$ for $B$ in $\cC$. Since $\cG\circ\iota\simeq\id_\cC$, the category $\cC$ is a full subcategory of $\cB$, and we have for every $A$ in $\cC$ that 
 $$ \Hom_\cC(A,B) \ \simeq \ \Hom_\cB(\iota(A),\iota(B)) \ \simeq \ \lim\; \Hom_\cB(\iota(A),\iota(\cD)) \ \simeq \ \lim\; \Hom_\cC(A,\cD), $$
 which shows that $B$ is the limit of $\cD$ in $\cC$ as desired. On the other hand, if $B=\lim\cD$, then by the adjointness of $\iota$ and $\cG$, we have for all $\tilde A$ in $\cB$ that
 $$ \Hom_\cB(\tilde A,\iota(B)) \ \simeq \ \Hom_\cC(\cG(\tilde A),B) \ \simeq \ \lim\; \Hom_\cC(\cG(\tilde A),\cD) \ \simeq \ \lim\; \Hom_\cB(\tilde A,\iota(\cD)), $$
 which means that $\iota(\lim\cD)\simeq\lim\iota(\cD)$.

 Assume that $A$ is the colimit $\colim\iota(\cD)$ in $\cB$. Again, since $\cC$ is a full subcategory of $\cB$ and by the adjointness of $\iota$ and $\cG$, we have for every $B$ in $\cC$ that 
 $$ \Hom_\cC(\cG(A),B) \simeq \ \Hom_\cB(A,\iota(B)) \ \simeq \ \colim\; \Hom_\cB(\iota(\cD),\iota(B)) \ \simeq \ \colim\; \Hom_\cC(\cD,B), $$
 which shows that $\cG(A)$ is the colimit of $\cD$ in $\cC$ as desired.
\end{proof}

%%%%%%%%%%%%%%%%%%%%%%%%%%%%%%%%%%%%%%%%%%%%%%%%%%%%%%%%%%%%%%%%%%%%%%%%%%%%%%%%%%%%%%%%%%%%%%%%%%%%%%%%%%%%%%%%%%%%%%%%%%%%%%%%%%%%%%%%%%%%%%%%%%%%%%%%%%%%%%%%%%%%%%%%%%%%%%%%%%%%%%%%%%%%%%%%%%%%%%%%%%%%%%%%%%%%%%%%%%%%

%\subsection{Examples}

\subsection{Monoids}
\label{section:monoids}

As a first class of examples of blueprints we consider the category of monoids, which we denote by $\cM$. It can be embedded as a full subcategory into $\bp$ in the following way. Let $A$ be a monoid, then $B=\bpgenquot A\emptyset$ is the blueprint with the smallest possible pre-addition, which contains only the relations $\sum a_i\=\sum a_i$ for $a_i\in A$. If $A_1\to A_2$ is a monoid morphism, then it extends to a morphism between the associated blueprints $\bpgenquot {A_1}\emptyset\to\bpgenquot {A_2}\emptyset$, and vice versa every blueprint morphism is a morphism between the underlying monoids. This establishes an embedding $\iota_\cM:\cM\hookrightarrow\bp$ of $\cM$ as a full subcategory of $\bp$. Note that the image of this inclusion lies in both $\bp_\canc$ and $\bp_\proper$. In the following, we call blueprints in the essential image of this embedding monoids. Note that a blueprint $B=\bpquot A\cR$ is a monoid if and only if $\cR=\gen\emptyset$. %The functor $\iota_\cM$ has a right-adjoint $\fA:\bp\to\cM$, which takes a blueprint $B=\bpquot A\cR$ to the quotient monoid $A_\proper=A/\sim$ where $a\sim b$ if $a\= b$ in $B$ (cf. the proof of Lemma \ref{lemma_proper}).

The interest in the category $\cM$ stems from Deitmar's theory of $\Fun$-schemes (cf.\ \cite{Deitmar05}), which are geometric objects associated to the category $\cM$. This theory will turn out to be a special case of the theory of blue schemes, as developed in this paper (cf.\ Section \ref{section:monoidal_schemes}). Connes and Consani defined the \emph{category $\cM_0$ of monoids with a zero} (or, an \emph{absorbing element}) and generalized Deitmar's constructions to $\cM_0$, (cf.\ \cite{CC09}). The terminology $\cM_0$-schemes established itself for the schemes associated to the category $\cM_0$; more theory and a large class of examples of $\cM_0$-schemes can be found in \cite{CLS10}. We show that also the category $\cM_0$ embeds into $\bp$ and that $\cM_0$-schemes are a special kind of blue schemes (cf.\ Section \ref{section:monoidal_schemes}).

We review the definition of $\cM_0$. A \emph{monoid with a zero} is a monoid $A$ with an element $0$ that has the property that $0\cdot a=0$ for all $a\in A$. A morphism between monoids with a zero is a multiplicative map that sends $1$ to $1$ and $0$ to $0$. To a monoid $A$ with a zero, we associate the blueprint $B=\bpgenquot A{0\=\zero}$. A morphism in $\cM_0$ is easily seen to define a morphism between the associated blueprints, and, vice versa, every morphism between the associated blueprints comes from a morphism between monoids with a zero. This establishes an embedding $\iota_{\cM_0}:\cM_0\hookrightarrow\bp$ of $\cM_0$ as full subcategory of $\bp$. Note that the image of this embedding lies in both $\bp_\canc$ and $\bp_0$. In the following, we call blueprints in the essential image of this embedding monoids with a zero. Note that a blueprint $B=\bpquot A\cR$ is a monoid with a zero if and only if $\cR=\gen{a\=\zero}$ for some $a\in A$.

The category $\cM_0$ proved to have certain advantages in contrast to $\cM$, in particular if one considers modules and exact sequences of modules (cf.\ \cite{CLS10}). For this reason, we define $\Fun$, the so-called \emph{field with one element}, as the initial object in $\cM_0$, which is the monoid $\{0,1\}$ with a zero $0\=\zero$. Note that this monoid with a zero, considered as a blueprint, is the initial object of $\bp_0$. The initial object of $\bp$, however, is equal to the initial object of $\cM$, which is the trivial monoid $\{1\}$.

\subsection{Semi-rings}
\label{section:semirings}

The next class of examples are blueprints that come from semirings. 

In this text, a \emph{semiring} has an associative and commutative addition with neutral element $0$ and an associative and commutative multiplication with neutral element $1$, which is distributive over the addition and such that $0\cdot a=0$ for any element $a$ of the semiring. A morphism of semirings is an additive and multiplicative map that maps $0$ to $0$ and $1$ to $1$. We denote the category of semirings by $\SRings$. To a semiring $R$, we associate the blueprint $\fB(R)=\bpquot A\cR$ where $A$ is the multiplicative monoid of $R$ and $\cR=\{\sum a_i\=\sum b_j|\sum a_i=\sum b_j\text{ in }R\}$. Then $\fB(R)$ is a proper blueprint with a zero. It is clear that a morphism between semirings induces a morphism between the associated blueprints. This defines an embedding $\SRings\to\bp_0$. In the following, we call blueprints in the essential image of this embedding semirings.

To show that $\SRings$ is a full subcategory of $\bp$, we characterize those blueprints that are semirings. Namely, it is easily seen that $B=\bpquot A\cR$ is a semiring if and only if $B$ is proper and with a zero and if for every $a,b\in A$, there is a $c\in A$ such that $a+b\= c$. In particular, the pre-addition $\cR$ of a semiring $\bpquot A\cR$ is generated by all relations of the form $a+b\= c$. From this it follows that a morphism $f: \bpquot {A_1}{\cR_1}\to \bpquot {A_2}{\cR_2}$ of blueprints that are semirings satisfies $f(a)+f(b)\= f(c)$ if $a+b\= c$. By properness, such a morphism comes from a morphism in $\SRings$, and $\SRings$ is thus a full subcategory of $\bp$.

The embedding $\fB:\SRings\to\bp$ has a left-adjoint functor $(\blanc)_\N:\bp\to\SRings$. If $B=\bpquot A\cR$ is a blueprint, then we define $B_\N$ as the quotient $\N[A]/\cR$. Note that by additivity and multiplicativity of the pre-addition $\cR$, the addition and multiplication of the quotient $B_\N=\N[A]/\cR$ is well-defined and makes $B_\N$ a semiring. The inclusion $A\to\N[A]$ defines a canonical morphism $\iota:B\to B_\N$ that satisfies the universal property that every morphism from $B$ into a semiring $R$ factors uniquely through $\iota:B\to B_\N$.

This implies that $\Hom_\SRings(B_\N,R)=\Hom_\bp(B,\fB(R))$ for every blueprint $B$ and every semiring $R$. Moreover, $\fB(R)_\N\simeq R$, which means that $(\blanc)_\N\circ\fB\simeq\id_\SRings$, and thus $\iota=\fB$ and $\cG=(\blanc)_\N$ satisfy the hypothesis of Proposition \ref{prop:adjoint_functors}.

\subsection{Rings}
\label{section:rings}

\begin{figure}[t]
 \begin{center}
  \includegraphics[width=0.7\textwidth,angle=0]{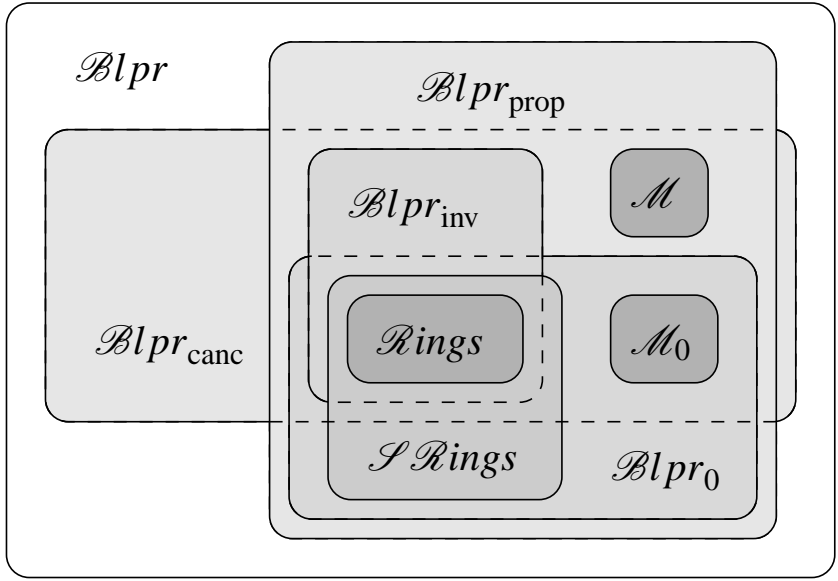}
  \caption{Overview of certain full subcategories of $\bp$}
  \label{figure:subcategories_of_bp}
 \end{center} 
\end{figure}

In this text, a \emph{ring} is always a commutative ring with $1$. Since the category $\Rings$ of rings is a full subcategory of $\SRings$, it can consequently be considered as a full subcategory of $\bp$. In the following, we call blueprints in the essential image of the embedding $\Rings\hookrightarrow\bp$ rings. A blueprint is a ring if and only if it is with inverses and satisfies the conditions of a semiring, i.e.\ it is proper, with a zero and for every $a,b\in B$, there is a $c\in B$ such that $a+b\= c$.

The embedding $\fB:\Rings\to\bp$ has a left-adjoint functor $(\blanc)_\Z:\bp\to\Rings$. If $B=\bpquot A\cR$ is a blueprint, then we define $B_\Z$ as the quotient $\Z[A]/\cI(\cR)$ where $\cI(\cR)=\{\sum a_i-\sum b_j|\sum a_i\=\sum b_j\text{ in }B\}$ is an ideal of the semi-group ring $\Z[A]$ by properties \textup{(A1)}--\textup{(A5)} of the pre-addition $\cR$. The inclusion $A\to\Z[A]$ defines a canonical morphism $\iota:B\to B_\Z$ that satisfies the universal property that every morphism from $B$ into a ring $R$ factors uniquely through $\iota:B\to B_\Z$.

This implies that $\Hom_\Rings(B_\Z,R)=\Hom_\bp(B,\fB(R))$ for every blueprint $B$ and every ring $R$. Moreover, $\cB(R)_\Z\simeq R$, which means that $(\blanc)_\Z\circ\fB\simeq\id_\SRings$, and thus $\iota=\fB$ and $\cG=(\blanc)_\Z$ satisfy the hypothesis of Proposition \ref{prop:adjoint_functors}. Note that $(B)_\N=(B)_\Z$ if $B$ is with inverses.

We summarize all the mentioned subcategories of $\bp$ in Figure \ref{figure:subcategories_of_bp}.

\subsection{Base extension from $\Fun$ to $\Z$}
\label{section:base_extension_from_f1_to_z}

One particular aspect of $\Fun$-geometry is that it is a geometry ``below $\Spec\Z$'', which means that a category $\Sch_\Fun$ of $\Fun$-schemes should come with a functor $\blanc\otimes_\Fun\Z:\Sch_\Fun\to\Sch_\Z$ to the category of (usual) schemes, which is called the \emph{base extension from $\Fun$ to $\Z$}. In the case of $\cM$-schemes (cf.\ \cite[Section 1.1]{LL09b}), this functor is defined in terms of affine coverings and the functor $\blanc\otimes_\Fun\Z:\cM\to\Rings$, which sends a monoid $A$ to the semi-group ring $\Z[A]$. This base extension functor is recovered in the theory of blueprints; namely, the diagram
$$ \xymatrix{\cM\ar[rrr]^{\iota_\cM} \ar[drrr]_{\blanc\otimes_\Fun\Z\quad}&&& \bp\ar[d]^{(\blanc)_\Z}\\ &&& \Rings} $$
commutes (up to equivalence). 

The variant for $\cM_0$-schemes (cf.\ \cite[Section 1.6]{LL09b}) is the functor $\blanc\otimes_\Fun\Z:\cM_0\to\Rings$, which send a monoid $A$ with a zero to the ring $\Z[A]/(0)$ where $(0)$ is the ideal generated by $0\in A\subset\Z[A]$. This base extension functor is recovered in the theory of blueprints in the same way; namely, the diagram
$$ \xymatrix{\cM_0\ar[rrr]^{\iota_{\cM_0}} \ar[drrr]_{\blanc\otimes_\Fun\Z\quad}&&& \bp\ar[d]^{(\blanc)_\Z}\\ &&& \Rings} $$
commutes (up to equivalence).

This means that both categories $\cM$ and $\cM_0$ together with their base extension functors $\blanc\otimes_\Fun\Z$ are special cases of blueprints and the functor $(\blanc)_\Z:\bp\to\Rings$. For the analogous geometric situation, cf.\ Section \ref{section:grothendieck_schemes}.

\subsection{Multiplicative maps into semirings}
\label{section:subsets_of_semirings}

We give a description of blueprints in terms of multiplicative maps from monoids into semirings.

Let $R$ be a semiring and $f:A\to R$ a multiplicative map from a monoid $A$ into $R$ that maps $1$ to $1$. A morphism from $f_1:A_1\to R_1$ to $f_2:A_2\to R_2$ is a monoid morphism $g:A_1\to A_2$ such that every equality $\sum f_1(a_i)=\sum f_1(b_j)$ in $R_1$ where $a_i,b_j\in A_1$ implies $\sum f_2(g(a_i))=\sum f_2(g(b_j))$ in $R_2$. Let $\cB$ be the category of all multiplicative maps $f:A\to R$ from a monoid $A$ into a semiring $R$.

Note that the semiring of an object $f:A\to R$ in $\cB$ plays a subordinated role: for instance, we can replace $R$ by a semiring $R'$ that contains $R$ and define $f':A\to R'$ as the composition of $f$ with the inclusion $R\hookrightarrow R'$. Then $f$ and $f'$ are isomorphic in $\cB$. In fact, only the sub-semiring of $R$ that is generated by $A$ is essential.

We can associate to each $f:A\to R$ in $\cB$ the blueprint $\fB(f)=\bpquot A{\cR_R}$ where $\cR_R$ consists of the relations $\sum a_i\=\sum b_j$ whenever $\sum f(a_i)=\sum f(b_j)$ holds in $R$ for $a_i,b_j\in A$. A morphism between objects in $\cB$ defines naturally a morphism between the associated blueprints; thus we obtain a functor $\fB:\cB\to \bp$.

Conversely, let $B=\bpquot A\cR$ be a blueprint. Then the canonical map $f:A\to A_\N$ is an object of $\cB$. A morphism of blueprints gives rise to a morphism in $\cB$. This defines a functor $\cF:\bp\to\cB$. The following is easy to see.

\begin{lemma}\label{lemma:semirings_and_blueprints}
 The functors 
 $$ \entrymodifiers={+!!<0pt,\fontdimen22\textfont2>}\xymatrix{{\hspace{1cm}\cB}\quad \ar@<-0.5ex>[rr]_\fB && \quad\bp\ar@<-0.5ex>[ll]_\cF} $$
are mutually inverse equivalences of categories. Let $B=\bpquot A\cR$ be a blueprint and $f:A\to A_\N$ the associated object $\cF(B)$ in $\cB$. Then we have the following characterizations:
 \begin{enumerate}
  \item $B$ is proper if and only if $f$ is injective; 
  \item $B$ is with a zero if and only if $0\in\im f$; 
  \item $B$ is cancellative if and only if $f$ is isomorphic to a multiplicative map $f':A\to R$ into a ring $R$; 
  \item $B$ is with inverses if and only if $A_\N$ is a ring.\qed
 \end{enumerate}
\end{lemma}

Note that the glued category $\mathfrak{MR}$ that Connes and Consani use in their definition of an $\Fun$-scheme (cf.\ \cite{CC09}) has a natural realization as the full subcategory of $\bp$ whose objects are either monoids with a zero or rings. For more details on the connection between Connes and Consani's $\Fun$-schemes and blue schemes (as developed in this paper), see Section \ref{section:CC-schemes}.

\subsection{Sesquiads}
\label{section:sesquiads}

We review Deitmar's theory of sesquiads (cf.\ \cite{Deitmar11}). A \emph{sesquiad} is a submonoid $A$ of a ring $R$ that contains $0\in R$ and that is endowed with a map $\Sigma_k:D_k\to A$ where $D_k=\{(a_i)\in A^n\subset R^n| b=\sum k_ia_i \in A\}$ and $\Sigma(a_i)=b$ for every $k=(k_1,\dotsc,k_n)\in\Z^n$ and every $n\geq2$. A \emph{sesquiad morphism} is a multiplicative map $\varphi:A_1\to A_2$ preserving $0$ and $1$ such that for all $(a_i)\in D_k\subset A^n$, we have $\varphi(\Sigma_k(a_i))=\Sigma_k(\varphi(a_i))$. We can associate to a sesquiad the blueprint $B=\bpquot A\cR$ where $\cR$ is the pre-addition generated by $\{\sum a_i\= b| \Sigma_k(a_i)=b\}$. This defines a fully faithful embedding of the category of sesquiads into the category of blueprints. The essential image of this embedding is the subcategory of cancellative proper blueprints with a zero.

Furthermore, congruences and ideals of sesquiads coincide with congruences and ideals of blueprints as introduced in Section \ref{section:congruences_and_ideals}. Consequently, a Zariski scheme in the sense of \cite{Deitmar11} is nothing else than a blue scheme (to be defined in Section \ref{section:blue_schemes}) that is locally isomorphic to spectra of sesquiads.

\subsection{$B_1$-algebras}
\label{section:b_1-algebras}

Following Lescot's paper \cite{Lescot09}, a $B_1$-algebra is a semiring $R$ with $1+1=1$. In particular, $B_1$ is the semiring $\{0,1\}$ with $1+1=1$ (all other sums and products are determined by the axioms of a semiring). The interest in $B_1$-algebras lies in a formal similarity with algebra over $\Fun$. Paul Lescot develops the notions of ideals, congruences and schemes.

As a semiring, a $B_1$-algebra is in particular a blueprint. Lescot's theory coincides with the theory of blueprints in spirit though there are some technical differences, e.g.\  in the definition of a (prime) congruence.

%%%%%%%%%%%%%%%%%%%%%%%%%%%%%%%%%%%%%%%%%%%%%%%%%%%%%%%%%%%%%%%%%%%%%%%%%%%%%%%%%%%%%%%%%%%%%%%%%%%%%%%%%%%%%%%%%%%%%%%%%%%%%%%%%%%%%%%%%%%%%%%%%%%%%%%%%%%%%%%%%%%%%%%%%%%%%%%%%%%%%%%%%%%%%%%%%%%%%%%%%%%%%%%%%%%%%%%%%%%%

\subsection{Cyclotomic field extensions of $\Fun$}
\label{section:cyclotomic_field_extensions}

Central objects in the theory of $\Fun$ are, besides $\Fun$ itself, the {cyclotomic field extensions $\Funn$ of $\Fun$}. The ``field'' $\Funn$ is often defined as the group $\mu_n$ of $n$-th roots of unity, or, as the monoid $\{0\}\cup\mu_n$ with a zero. However, these definitions have the following disadvantages: there are morphisms $\Funn\to\Fun$, which are not injective and thus untypical for ``field homomorphisms''; and the extension of $\mu_n$ to $\Z$ is the semi-group ring $\Z[\mu_n]$, which does not equal the ring of integers of the cyclotomic field extension $\Q[\mu_n]$ of $\Q$ unless $n=1$. Therefore we propose the following alternative definition.

The \emph{$n$-th cyclotomic field extension $\Funn$ of $\Fun$} is the blueprint $B=\bpquot A\cR$ where $A=\{0\}\cup\mu_n$ and $\cR$ is generated by the equation $0\= \zero$ and $\sum_{h\in H} h\=0$ for all non-trivial subgroups $H$ of $\mu_n$. The blueprint $\Funn$ is proper, with a zero and cancellative for all $n\geq1$, and it is with inverses if and only if $n$ is even.

This definition satisfies the desired properties mentioned above: every morphism $\F_{1^k}\to\Funn$ is injective. Consequently, $\mu_k$ is a subgroup of $\mu_n$ and $k$ is a divisor of $n$.

The ``base extension to $\Z$'', i.e.\ the generated ring $\cO_n=(\Funn)_\Z$, is isomorphic to the ring of integers of the cyclotomic number field $\Q[\mu_n]$. Note that $(\Funn)_\Z\simeq(\F_{1^{2n}})_\Z$ if $n$ is odd. Note further that if $n=q-1$ where $q=p^a$ is the $a$-th power of a prime $p$, then the inertia degree of a prime $\fp$ over $p$ in $\cO_n$ is $a$, which means that $\cO_n/\fp\simeq\F_q$. Thus we have a morphism of blueprints $\Funn\to\F_q$, which is an isomorphism between the underlying monoids.

%%%%%%%%%%%%%%%%%%%%%%%%%%%%%%%%%%%%%%%%%%%%%%%%%%%%%%%%%%%%%%%%%%%%%%%%%%%%%%%%%%%%%%%%%%%%%%%%%%%%%%%%%%%%%%%%%%%%%%%%%%%%%%%%%%%%%%%%%%%%%%%%%%%%%%%%%%%%%%%%%%%%%%%%%%%%%%%%%%%%%%%%%%%%%%%%%%%%%%%%%%%%%%%%%%%%%%%%%%%%

\subsection{Archimedean valuation rings}
\label{section:archimedean_valuation_rings}

Blueprints bridge the gap between archimedean and non-archimedean absolute values. 

Let $\norm\ :K\to\R_{\geq0}$ be an absolute value of a number field $K$ with integers $\cO_K$. Then $A=\{a\in K|\norm a\leq1\}$ is a submonoid of $K$. If $\norm\ $ is non-archimedean, the valuation ring of $\norm\ $ corresponds to the blueprint $\fB(A,K)$ (by the means of Lemma \ref{lemma:semirings_and_blueprints}). A first attempt is to define the ``archimedean valuation ring'' as $\fB(A,K)$. However, the only ideal of this blueprint is the zero ideal (for the definition of ideals, see Section \ref{subsection:ideals}). A better definition is $B=\bpquot A\cR$ where $\cR$ is the pre-addition generated by the relations $0\=\zero$ and $\sum_{a\in H}a\=0$ for all finite non-trivial subgroups $H\subset A^\times$. This is a proper blueprint with inverses and a zero. It has a unique maximal ideal, which is $\fm=\{a\in A|\norm a<1\}$ (cf.\ Section \ref{section:maximal_congruences_and_maximal_ideals}), and the quotient $A/\fm$ is isomorphic to $\F_{1^n}[T_1^{\pm1},\dotsc,T_k^{\pm1}]$ where $n$ is the number of roots of unity in $K$ and $k$ is the rank of the group $\cO_K^\times$ of $K$-units (the blueprint $\F_{1^n}[T_1^{\pm1},\dotsc,T_k^{\pm1}]$ is the localization of the free blueprint $B=\F_{1^n}[T_1,\dotsc,T_k]$ over $\Funn$ at $S=B-\{0\}$, cf.\ Sections \ref{section:free_algebras} and \ref{section:localizations} for definitions).

%%%%%%%%%%%%%%%%%%%%%%%%%%%%%%%%%%%%%%%%%%%%%%%%%%%%%%%%%%%%%%%%%%%%%%%%%%%%%%%%%%%%%%%%%%%%%%%%%%%%%%%%%%%%%%%%%%%%%%%%%%%%%%%%%%%%%%%%%%%%%%%%%%%%%%%%%%%%%%%%%%%%%%%%%%%%%%%%%%%%%%%%%%%%%%%%%%%%%%%%%%%%%%%%%%%%%%%%%%%%

\subsection{Free algebras}
\label{section:free_algebras}

Given a blueprint $B=\bpquot A\cR$ and a set $S=\{T_i\}_{i\in I}$, we define the \emph{free blueprint over $B$ generated by $S$} as the blueprint $B[S]=\bpquot{A[S]}{\cR[S]}$ where 
\[
 A[S] \quad = \quad \Bigr\{ \ a \prod_{i\in I} T_i^{n_i} \ \Bigr| \ a\in B,\ (n_i)\in\bigoplus_{i\in I} \N \ \Bigr\},
\]
which is a monoid with respect to the multiplication $(a\prod T_i^{n_i})\cdot(b\prod T_i^{m_i})=ab\prod T_i^{n_i+m_i}$, and $\cR[S]$ is the pre-addition on $A[S]$ generated by
\[
 \Bigr\{\ \sum a_{k}  \prod T_i^{n_{i}} \= \sum b_{l} \prod T_i^{n_{i}} \ \Bigr| \ \sum a_k \=\sum b_k \ \Bigr\}.
\]
It satisfies the universal property of the free object over $S$: for every map $f:S\to C$ into a blueprint $C$, there is a unique morphism $g: B[S]\to C$ such that $g(T_i)=f(T_i)$ for every $i\in I$. We also write $B[T_i]_{i\in I}$ for $B[S]$ or $B[T_1,\dotsc,T_n]$ if $S=\{T_1,\dotsc,T_n\}$.

If $B$ is contained in $\bp_\canc$ or $\cM$, then $B[S]$ is so, too. If, however, $B$ contains a relation of the form $\sum a_k\=\zero$ with a non-trivial sum $\sum a_k$, then $B[S]$ is not proper since $\sum a_k \prod T_i^{n_i} \=\sum a_k\prod T_i^{m_i}$, independently of $(n_i)$ and $(m_i)$. One sees easily that if $B$ is contained in one of $\bp_\proper$, $\bp_\inv$, $\bp_0$ or $\cM_0$, then $B[S]_\proper$ is so, too, and it is the free object over $S$ in the corresponding subcategory. If $B$ is a ring, then $B[S]$ is not a ring unless $S$ is empty, but $B[S]_\Z$ is the free blueprint over $R$ in the sense of rings. Similarly, if $B$ is a semiring, then $B[S]_\N$ is the free semiring over $B$ generated by $S$.

%%%%%%%%%%%%%%%%%%%%%%%%%%%%%%%%%%%%%%%%%%%%%%%%%%%%%%%%%%%%%%%%%%%%%%%%%%%%%%%%%%%%%%%%%%%%%%%%%%%%%%%%%%%%%%%%%%%%%%%%%%%%%%%%%%%%%%%%%%%%%%%%%%%%%%%%%%%%%%%%%%%%%%%%%%%%%%%%%%%%%%%%%%%%%%%%%%%%%%%%%%%%%%%%%%%%%%%%%%%%

\subsection{Localizations}
\label{section:localizations}

 Let $B=\bpquot A\cR$ be a blueprint and $S$ a \emph{multiplicative subset} of $B$, i.e.\ a submonoid of $A$. We define $S^{-1}A$ as the quotient of $A\times S$ by the equivalence relation $\sim$ given by $(a,s)\sim(a',s')$ if and only if there is a $t \in S$ such that $tsa'=ts'a$. We write $\frac as$ for the equivalence class of $(a,s)$ in $S^{-1}A$. We define $S^{-1}\cR$ as the set 
 \[
  S^{-1}\cR \quad = \quad \Bigr\{\ \sum\frac{a_i}{s_i}\=\sum\frac{b_j}{r_j}\ \Bigl|\text{ there is a }t\in S\text{ such that }\sum ts^ia_i\=\sum tr^jb_j\ \Bigl\} 
 \]
 \[
  \hspace{6.5cm}\text{ where }\quad s^i=\prod_{k\neq i}s_k\cdot\prod_j r_j\quad \text{ and }\quad r^j=\prod_is_i\cdot\prod_{l\neq j} r_l.
 \]
 We leave the elementary proofs of that $S^{-1}A$ is a monoid (with the multiplication inherited from $A\times S$) and that $S^{-1}\cR$ is a pre-addition for $S^{-1}A$ to the reader. We define the \emph{localization of $B$ at $S$} as the blueprint $S^{-1}B=\bpquot{S^{-1}A}{S^{-1}\cR}$.

 The association $a\mapsto \frac a1$ defines an epimorphism $B\to S^{-1}B$. It satisfies the universal property that every morphism $f:B\to C$ that maps $S$ to the units of $C$ factors uniquely through $B\to S^{-1}B$.

 It easy to see that if $B$ is contained in one of the subcategories $\bp_\canc$, $\bp_\proper$, $\bp_\inv$, $\bp_0$, $\cM$, $\cM_0$, $\SRings$ or $\Rings$, then $S^{-1}B$ is so, too.

\subsection{Limits and colimits}
\label{section:limits-colimits}

We prove the existence of small limits as well as finite and directed colimits. Recall that $\bp$ has an initial object, namely, the blueprint $\bpgenquot{\{1\}}{\emptyset}$, and a terminal object, namely, the trivial blueprint $\bpgenquot{\{1\}}{1\=\zero}$. The trivial blueprint is contained in all the subcategories $\bp_\canc$, $\bp_\proper$, $\bp_\inv$, $\bp_0$, $\cM$, $\cM_0$, $\SRings$ and $\Rings$, and is thus a terminal object of these subcategories. The initial object, which we denote simply by $\{1\}$, is contained in $\bp_\canc$, $\bp_\proper$, $\bp_\inv$ and $\cM$. The initial object of $\cM_0$ and $\bp_0$ is the blueprint associated to the monoid $\{0,1\}$ with zero $0$, the initial object of $\SRings$ is the blueprint associated to the semiring $\N$ and the initial object of $\Rings$ is the blueprint associated to the ring $\Z$.

Recall that the inclusions $\iota$ of $\bp_\canc$, $\bp_\proper$, $\bp_\inv$, $\bp_0$, $\SRings$ and $\Rings$ into $\bp$ have a left-adjoint $\cG$. So we can make use of Proposition \ref{prop:adjoint_functors} in the following proofs to describe limits and colimits in these subcategories in terms of the limits and colimits in $\bp$.

\begin{prop} \label{prop:limits_in_bp}
 The category $\bp$ contains small limits. Small limits in all the subcategories $\bp_\canc$, $\bp_\proper$, $\bp_\inv$, $\bp_0$, $\SRings$ and $\Rings$ exist and coincide with the small limits in $\bp$.
\end{prop}

\begin{proof}
 To prove that $\bp$ contains small limits, it suffices to prove that $\bp$ contains small products and equalizers (cf.\ \cite[Thm.\ 2.8.1]{Borceux94}). The product of a family of blueprints $\{B_i\}_{i\in I}$ is given by the Cartesian product $\prod B_i$ over $I$ together with componentwise multiplication and the pre-addition 
 $$ \cR \quad = \quad \{\ \sum(c_{i,k})_{i\in I}\=\sum (d_{i,l})_{i\in I} \ | \ \sum c_{i,k}\=\sum d_{i,l} \text{ for all }i\in I\ \}. $$
 The canonical projections of the product are the componentwise projection $p_j:\prod B_i\to B_j$. The universal property of a product is verified immediately for $\prod B_i$. If $\cC$ is one of the categories $\bp_\canc$, $\bp_\proper$, $\bp_\inv$, $\bp_0$, $\SRings$ or $\Rings$ and all the $B_i$ are in $\cC$, then the product is also an object in $\cC$ and satisfies the universal property of the product in $\cC$.

 The equalizer of two morphisms $f,g:B\to C$ is the subblueprint $\eq(f,g)=\bpquot A\cR$ where $A$ is the monoid $\{a\in B\mid f(a)=g(a)\}$ and $\cR$ is the restriction of the pre-addition of $B$ to $A$. Since $f(1)=1=g(1)$, the equalizer contains $1$ and since $f(ab)=f(a)f(b)=g(a)g(b)=g(ab)$ for all $a,b\in A$, the set $A$ is multiplicatively closed and thus indeed a monoid. It is clear that the natural inclusion $\eq(f,g)\to B$ satisfies the universal property of the equalizer of $f$ and $g$. It is also easily verified for all subcategories $\cC$ as in the theorem that $\eq(f,g)$ is in $\cC$ if $B$ and $C$ are so.
\end{proof}

\begin{rem}
 Note that there is a digression between the categorical product in $\bp$ and in the subcategories $\cM$ and $\cM_0$ as the following basic example shows. Let $A=\{e,1\}$ be a monoid where $e^2=e$ is an idempotent element and let $B=\bpgenquot A\emptyset$ be the blueprint associated to $A$. Then the product $B\times B$ is the monoid $A\times A=\{(e,e),(e,1),(1,e),(1,1)\}$ together with the pre-addition generated by $(e,1)+(1,e)\=(e,e)+(1,1)$. Consequently $B\times B$ is not a monoid and thus in particular not the blueprint associated to the monoid $A\times A$. The analogous example for the monoid $A=\{0,e,1\}$ with zero $0$ shows that also products in $\cM_0$ fail to coincide with products in $\bp$.
\end{rem}

\begin{prop}
 Let $\gamma:B\to C$ and $\delta:B\to D$ be morphisms of blueprints. Then the tensor product $C\otimes_B D$ exists in $\bp$. If $\cC$ is one of the categories $\cM$ or $\cM_0$ and if $B$, $C$ and $D$ are in $\cC$, then $C\otimes_BD$ is in $\cC$ and represents the tensor product in $\cC$. If $\cC$ is one of the categories $\bp_\canc$, $\bp_\proper$, $\bp_0$, $\bp_\inv$, $\SRings$ or $\Rings$ and $\cG: \bp\to \cC$ is the left-adjoint to the usual embedding $\iota:\cC\to\bp$, then the tensor product in $\cC$ equals $\cG(C\otimes_BD)$ provided $B$, $C$, and $D$ are in $\cC$.
\end{prop}

\begin{proof}
 We construct the tensor product $C\otimes_BD$ and prove that it represents the colimit of the diagram 
 $$ \xymatrix@C=1pc@R=1,5pc{C \\ B\ar[u]^{\gamma}\ar[rr]_{\delta}&&D} $$
 in the following. For $b\in B$, $c\in C$ and $d\in D$, we write $b.c$ and $b.d$ for the product $\gamma(b)\cdot c$ in $C$ resp.\ $\delta(b)\cdot d$ in D. Let $\sim$ be the equivalence relation on $C\times D$ that is generated by the relations $(b.c,d)\sim(c,b.d)$ where $b\in B$, $c\in C$ and $d\in D$. The quotient $A=B\times C/\sim$ inherits a multiplication from the product multiplication on $B\times C$ since $(b.c,d)\cdot(c',d')=(b.cc',dd')\sim(cc',b.dd')=(c,b.d)\cdot(c',d')$ for $b\in B$, $c,c'\in C$ and $d,d'\in D$. We write classes $[(c,d)]$in $A$ as $c\otimes d$ and endow $A$ with the pre-addition $\cR$ that is generated by the relations $\sum c_k\otimes 1\=\sum c'_l\otimes 1$ if $\sum c_k\=\sum c'_l$ in $C$ and $\sum 1\otimes d_k\=\sum 1\otimes d'_l$ if $\sum d_k\=\sum d'_l$ in $D$. We denote the blueprint $\bpquot A\cR$ by $C\otimes_BD$. It comes together with the morphisms $\iota_C:C\to C\otimes_BD$ with $\iota_C(c)=c\otimes 1$ and $\iota_D:D\to C\otimes_BD$ with $\iota_D(d)=1\otimes d$. Note that $\iota_C\circ\gamma=\iota_D\circ\delta$.

 We prove that $C\otimes_BD$ together with $\iota_C$ and $\iota_D$ satisfies the universal property of the tensor product. Let 
 $$ \xymatrix@C=1pc@R=1,5pc{C \ar[rr]^f && E \\ B\ar[u]^{\gamma}\ar[rr]_{\delta}&&D\ar[u]_{g}} $$
 be a commutative diagram. We have to show that there is a unique morphism $h:C\otimes_BD\to E$ such that $h\circ\iota_C=f$ and $h\circ\iota_D=g$. Since morphisms of blueprints are in particular multiplicative maps, it is clear that the only possible definition of $h$ is $h(c\otimes d)=f(c)\cdot g(d)$. Then $h$ is indeed multiplicative and we have $h\circ\iota_C=f$ and $h\circ\iota_D=g$ as multiplicative maps. For to show that $h$ maps $\cR$ to the pre-addition on $E$, it satisfies to verify this condition for generators of $\cR$. Let $\sum c_k\otimes 1\=\sum c'_l\otimes 1$ with $\sum c_k\=\sum c'_l$ in $C$. Then we can use that $f$ is a morphism of blueprints to verify that 
 $$ \sum h(c_k\otimes 1)\ =\ \sum f(c_k)\cdot g(1) \ =\ \sum f(c_k)\  \=\ \sum f(c'_l) \ = \ \sum f(c'_l)\cdot g(1) \ = \ \sum h(c'_l\otimes 1). $$\
 An analogous argument shows that $ \sum h(c_k\otimes 1)\=\sum h(c'_l\otimes 1)$, and thus $C\otimes_BD$ represents indeed the tensor product of $C$ and $D$ over $B$.

 It is easily verified that $C\otimes_BD$ is in $\cM$ or $\cM_0$ if $B$, $C$ and $D$ are so. The statement about the tensor product in $\bp_\canc$, $\bp_\proper$, $\bp_0$, $\bp_\inv$, $\SRings$ and $\Rings$ follows from Proposition \ref{prop:adjoint_functors}.
\end{proof}

Since $\bp$ has an initial object, the existence of finite colimits follows from the existence of tensor products.

\begin{cor}\label{cor:finite_colimits_in_bp}
 The category $\bp$ contains finite colimits. \qed
\end{cor}

Recall that a directed diagram in a category is a commutative diagram $\cD$ where for every pair of objects $B_i$ and $B_j$ in $\cD$, there are an object $B_k$ and morphisms $B_i\to B_k$ and $B_j\to B_k$ in $\cD$.

\begin{prop}\label{prop:colimits_of_directed_diagrams_in_bp}
 The category $\bp$ contains colimits of directed diagrams. If $\cD$ is a directed diagram in one of the subcategories $\cM$, $\cM_0$, $\bp_\canc$, $\bp_\proper$, $\bp_\inv$, $\bp_0$, $\SRings$ or $\Rings$, then the colimit of $\cD$ in $\bp$ is also in this subcategory of $\bp$ and equals the colimit of $\cD$ as a diagram in this subcategory.
 \end{prop}

\begin{proof}
 Let $\cD=\{B_i\}_{i\in I}$ be a commutative diagram of blueprints and morphisms indexed by a directed set $I$, i.e.\ for every $i,j\in I$, there is a $k\in I$ and (unique) morphisms $f_i:B_i\to B_k$ and $f_j:B_j\to B_k$ in $\cD$. For $i\in I$, define $J(i)$ as the cofinal directed subset $\{k\in I\mid \exists f:B_i\to B_k\text{ in }\cD\}$ of $I$, and let $\cD(i)$ be the full subdiagram of $\cD$ that contains precisely $\{B_i\}_{i\in J(i)}$. Then the colimit of $\cD$ can be represented by
 $$ \colim \cD \quad = \quad \coprod_{i\in I}\ \biggl\{ (a_j)\in \prod_{j\in J(i)}B_j\biggl| \forall f: B_j\to B_k\text{ in } \cD(i),\, a_k=f(a_j)\biggr\}\; \biggl/ \; \sim $$ 
 where two elements $(a_j)_{j\in J(i_1)}$ and  $(b_j)_{j\in J(i_2)}$ are equivalent if $a_j=b_j$ for all $j\in J(i_1)\cap J(i_2)$. The pre-addition for $\colim \cD$ is defined by $\sum(c_{j,k})_{j\in J(i(k))}\=\sum(d_{j,l})_{j\in J(i(l))}$ if and only if there is a $j$ in the intersection of the $J(i(k))$ and $J(i(l))$ such that $\sum c_{j,k}\=\sum d_{j,l}$.

 The canonical morphisms $\iota_i:B_i\to\colim \cD$ maps $a_i \in B_i$ to $(f(a_i)\mid f:B_i\to B_k\text{ in }\cD(i))$. Given a family of monoid morphisms $g_i:B_i\to C$ that commute with all morphisms in $\cD$, the map $\displaystyle g:\colim \cD \to B$ that sends an element $(a_j)_{j\in J(i)}$ to $g_i(a_i)$ is the unique morphism that satisfies the universal property of the colimit of $\cD$.  

 It is easily verified that the defining properties of the subcategories $\cM$, $\cM_0$, $\bp_\canc$, $\bp_\proper$, $\bp_\inv$, $\bp_0$, $\SRings$ and $\Rings$ are preserved under the above construction. Thus we have the second statement of the proposition.
\end{proof}

%%%%%%%%%%%%%%%%%%%%%%%%%%%%%%%%%%%%%%%%%%%%%%%%%%%%%%%%%%%%%%%%%%%%%%%%%%%%%%%%%%%%%%%%%%%%%%%%%%%%%%%%%%%%%%%%%%%%%%%%%%%%%%%%%%%%%%%%%%%%%%%%%%%%%%%%%%%%%%%%%%%%%%%%%%%%%%%%%%%%%%%%%%%%%%%%%%%%%%%%%%%%%%%%%%%%%%%%%%%%
%%%%%%%%%%%%%%%%%%%%%%%%%%%%%%%%%%%%%%%%%%%%%%%%%%%%%%%%%%%%%%%%%%%%%%%%%%%%%%%%%%%%%%%%%%%%%%%%%%%%%%%%%%%%%%%%%%%%%%%%%%%%%%%%%%%%%%%%%%%%%%%%%%%%%%%%%%%%%%%%%%%%%%%%%%%%%%%%%%%%%%%%%%%%%%%%%%%%%%%%%%%%%%%%%%%%%%%%%%%%

\section{Congruences and ideals}
\label{section:congruences_and_ideals}

In this section, we introduce the notions of congruences and ideals. Since in the category of blueprints, like in the category of rings, the initial object is not isomorphic to the terminal object, we cannot define a categorical kernel as the difference kernel with the zero morphism, but have to invent a notion that represents quotients of blueprints. This desire will be fulfilled by congruences, which can be seen as a generalization of congruences of monoids. Congruences are accompanied by their absorbing ideals, which are sets of elements that map to a zero or, more generally, to an absorbing element (to be defined below) of the quotient. This leads to the definition of an ideal as an absorbing ideal of a congruence.

%%%%%%%%%%%%%%%%%%%%%%%%%%%%%%%%%%%%%%%%%%%%%%%%%%%%%%%%%%%%%%%%%%%%%%%%%%%%%%%%%%%%%%%%%%%%%%%%%%%%%%%%%%%%%%%%%%%%%%%%%%%%%%%%%%%%%%%%%%%%%%%%%%%%%%%%%%%%%%%%%%%%%%%%%%%%%%%%%%%%%%%%%%%%%%%%%%%%%%%%%%%%%%%%%%%%%%%%%%%%

\subsection{Congruences}
\label{subsection:congruences}

%\begin{pg}\label{pg:subblueprints_and_quotients}
 Throughout this section, we let $B=\bpquot A\cR$ be a blueprint. We begin with introducing some preliminary notions. A \emph{subblueprint of $B$} is a submonoid $A'$ of $A$ together with the restriction $\cR\vert_{A'}$ of $\cR$ to $A'$, i.e.\ the pre-addition
 $$\cR\vert_{A'} \quad =\quad \{ \ \sum a_i\=\sum b_j \ | \ a_i,b_j\in A' \ \}$$
 for $A'$. We call $\bpquot{A'}{\cR\vert_{A'}}$ the \emph{subblueprint over $A'$}. It satisfies the universal property that every morphism $C\to B$ whose image is contained in $A'$ factorizes uniquely through the natural inclusion $\bpquot{A'}{\cR\vert_{A'}}\hookrightarrow B$. 

 We define the \emph{integral subblueprint of $B$} as the subblueprint over $A^\int$, which is the multiplicative subset of \emph{integral elements of $A$}, i.e.\ the subset of all $a\in A$ such that multiplication by $a$ defines an injective map $A\to A$. We say that $B$ is an \emph{integral blueprint} if $B$ is proper, if $1\n=\zero$ and if every $a\in B$ is either integral or a zero. A blueprint $B$ is said to be \emph{without zero divisors} if the set $S=\{a\in B|a\n=\zero\}$ is a multiplicative subset of $B$. Every integral blueprint is without zero divisors, but the example of the monoid $B=\{0,e,1\}$ with a zero $0\=\zero$ and an idempotent $e^2=e$ shows that blueprints without zero divisors do not have to be integral.

 We define the \emph{units of $B$} as the subblueprint over $A^\times$, which is the group of invertible elements of $A$. We say that $B$ is a \emph{blue field} if $B$ is proper, if $1\n=\zero$ and if every $a\in B$ is either a unit or a zero.  If $B$ is without zero divisors, then we define the \emph{quotient field $\Quot B$ of $B$} as $S^{-1}B$ where $S=\{a\in B|a\n=\zero\}$. It is clear that $\Quot B$ is a blue field.

 A \emph{quotient of $B=\bpquot A\cR$} is a blueprint $B'=\bpquot{A'}{\cR'}$ together with a morphism $B\to B'$ such that the map $A\to A'$ between the underlying monoids is surjective and such that the image of $\cR'$ in $\cR$ generates $\cR$ as a pre-addition. Note that the word ``quotient'' is used with different meanings: the quotient field of an integral blueprint $B$ is not a quotient of $B$ unless $B$ is a blue field.

 If $f: A\to A'$ is a surjective morphism of monoids, then $\cR'=\gen{\{\sum f(a_i)\=\sum f(b_j)|\sum a_i=\sum b_j\text{ in }B\}}$ is called the \emph{quotient pre-addition for $A'$}. The blueprint $B'=\bpquot{A'}{\cR'}$ together with the induced morphism $B\to B'$ is the unique quotient whose map between the underlying monoids is $A\to A'$. Thus quotients of $B$ are characterized by surjections from $A$ into other monoids. 

 If $B$ is with inverses or with a zero, then a quotient of $B$ has the same properties. This is in general not true for cancellative or proper blueprints. However, we will be mainly interested in proper quotients of blueprints. To characterize the proper quotients of a blueprint, we introduce the notion of a congruence.

 For an equivalence relation $\sim$ on $A$, we define the \emph{linear extension $\sim_\N$ of $\sim$ to $\N[A]$} as the equivalence relation on $\N[A]$ that is generated by $\sum a_i \sim_\N\sum b_i$ if $a_i\sim b_i$ for all $i$. Further define the equivalence relation $\sim_\cR$ as the smallest equivalence relation containing both $\cR$ and $\sim_\N$.
%\end{pg}

\begin{df}\label{def:congruence}
 A \emph{congruence on $B$} is an equivalence relation $\sim$ on $A$ that satisfies the following properties.
 \begin{enumerate}
  \item[(C1)] The equivalence relation $\sim_\N$ is a pre-addition.
  \item[(C2)] The restriction of $\sim_\cR$ to $A$ equals $\sim$.
 \end{enumerate}
\end{df}

These axioms are equivalent to the following explicit conditions.

\begin{lemma}\label{lemma:equivalent_definition_of_a_congruence}
 An equivalence relation $\sim$ is a congruence if and only if it satisfies the following two conditions for all $a,b,c,d\in A$.
 \begin{enumerate}
  \item[\textup{(C1)$^*$}\!\!] \ If $a\sim b$ and $c\sim d$, then $ac\sim bd$.
  \item[\textup{(C2)$^*$}\!\!] \ If there exists a sequence
        $$ a \ \= \ \sum c_{1,k} \ \sim_\N \ \sum d_{1,k} \ \= \ \sum c_{2,k} \ \sim_\N \quad \dotsb \quad \sim_\N \ \sum d_{n,k} \ \= \ b $$
        with $c_{i,k},d_{i,k}\in A$, then $a\sim b$.
 \end{enumerate}
 More precisely, axiom \textup{(C1)} is equivalent to condition \textup{(C1)}$^*$ and axiom \textup{(C2)} is equivalent to condition \textup{(C2)}$^*$.
\end{lemma}

\begin{proof}
 We begin with the equivalence of axiom \textup{(C1)} with condition \textup{(C1)}$^*$. Since the restriction of $\sim_\N$ to $A$ is $\sim$, it is clear that \textup{(C1)}$^*$ follows from the multiplicativity of the pre-addition $\sim_\N$. Conversely, assume \textup{(C1)}$^*$. Note that $\sim_\N$ satisfies axioms \textup{(A1)}--\textup{(A4)} by its definition. Thus we have only to verify the multiplicativity of $\sim_\N$. Let $\sum a_i\sim_\N\sum b_j$ and $\sum c_k\sim_\N\sum d_l$. Then by \textup{(C1)}$^*$, $a_ic_k\sim b_jd_l$ and thus $\sum a_ic_k\sim_\N\sum b_jd_l$, which shows axiom \textup{(C1)} of a congruence.

 We proceed with the equivalence of axiom \textup{(C2)} with condition \textup{(C2)}$^*$. By the definition of $\sim_\cR$, we have that $a\sim_\cR b$ if and only if there exists a sequence of the form
 $$ a \ \= \ \sum c_{1,k} \ \sim_\N \ \sum d_{1,k} \ \= \ \sum c_{2,k} \ \sim_\N \quad \dotsb \quad \sim_\N \ \sum d_{n,k} \ \= \ b. $$
 Given a sequence as above, this means that $a\sim_\cR b$, and thus, by axiom \textup{(C2)}, that $a\sim b$. This shows that \textup{(C2)} implies \textup{(C2)}$^*$. Notice that $\sim_\cR$ contains $\sim$ by its definition. Thus to show the converse implication, we have to deduce $a\sim_\cR b$ from $a\sim b$. But this is clear: if $a\sim_\cR b$, then there is a sequence as above and thus $a\sim b$ by \textup{(C2)}$^*$. 
\end{proof}

\begin{df}
 Let $f:B\to C$ be a morphism of blueprints. The \emph{kernel of $f$} is the relation $\sim_f$ on $B$ that is defined by $a\sim_f b$ if and only if $f(a)\= f(b)$.
\end{df}

The following two propositions show that kernels are congruences and, conversely, that congruences are kernels of a quotient map.

\begin{prop}\label{prop:kernels_are_congruences}
 Let $f:B\to C$ be a morphism of blueprints. Then its kernel $\sim_f$ is a congruence.
\end{prop}

\begin{proof}
 Since $\=$ is an equivalence relation, $\sim_f$ is so, too. We show \textup{(C1)}$^*$. Let $a\sim_f b$ and $c\sim_f d$, i.e.\ $f(a)\= f(b)$ and $f(c)\= f(d)$. Then $f(ac)\= f(a)f(c)\= f(b) f(d)\=f(bd)$ by the multiplicativity of $f$ and $\=$, and thus $ac\sim_f bd$.
 
 We show \textup{(C2)}$^*$. Consider a sequence 
 $$ a \ \= \ \sum c_{1,k} \ \sim_{f,\N} \quad \dotsb \quad \sim_{f,\N} \ \sum d_{n,k} \ \= \ b, $$
 then it follows that 
 $$ f(a) \ \= \ \sum f(c_{1,k}) \ \= \quad \dotsb \quad \= \ \sum f(d_{n,k}) \ \= \ f(b) $$
 and thus $f(a)\= f(b)$. This shows that $a\sim_f b$ as desired.
\end{proof}

Conversely, we can divide by congruences. Let $\sim$ be a congruence of $B$. Then we define the quotient $B/\sim$ as the blueprint $\bpquot{(A/\sim)}{\cR_\sim}$ where $A/\sim$ is the quotient monoid and $\cR_\sim$ is the pre-addition generated by $\{\sum[a_i]\=\sum[b_j]| \sum a_i\=\sum b_j\text{ in }B\}$. Note that $A/\sim$ is indeed a monoid since condition \textup{(C1)}$^*$ guarantees that the multiplication on equivalence classes is well-defined.

\begin{prop}\label{prop:congruences_are_kernels}
 Let $\sim$ be a congruence on $B$ and let $p:B\to B/\sim$ be the canonical projection. Then $\sum[a_i]=\sum[b_j]$ if and only if there is a sequence 
 $$ \sum a_i\ \=\ \sum c_{1,k} \ \sim_\N \quad \dotsb \quad \sim_\N\ \sum d_{n,k} \ \= \ \sum b_j. $$
 Consequently, $B/\sim$ is proper and $\sim$ is the kernel of $p$.
\end{prop}

\begin{proof}
 Since $\cR_\sim$ is transitive, we have $\sum[a_i]\=\sum[b_j]$ if there is a sequence of the form 
 $$ \sum a_i\ \=\ \sum c_{1,k} \ \sim_\N \quad \dotsb \quad \sim_\N \ \sum d_{n,k} \ \= \ \sum b_j. $$
 If we show that the set $\cR_\sim'$ of relations $\sum[a_i]\=\sum[b_j]$ that come from such a sequence form a pre-addition, then it equals $\cR_\sim$ since $\cR_\sim$ is generated by a subset of $\cR_\sim'$. It is clear that $\cR_\sim'$ is an equivalence relation. So we are left with showing additivity and multiplicativity.

 Consider two sequences
 $$ \sum a_i\ \=\ \sum c_{1,k} \ \sim_\N \quad \dotsb \quad \sim_\N \ \sum d_{n,k} \ \= \ \sum b_j $$
 and 
 $$ \sum \tilde a_{\tilde i}\ \=\ \sum \tilde c_{1,\tilde k} \ \sim_\N  \quad \dotsb \quad \sim_\N \ \sum \tilde d_{\tilde n,\tilde k} \ \= \ \sum \tilde b_{\tilde j}. $$
 Since we have the trivial relations $\sum a_i\=\sum a_i$ and $\sum a_i\sim_\N\sum a_i$, we may insert some trivial relations to make sure that the two sequences have the same length, this means, we may assume that $\tilde n=n$. Since both $\=$ and $\sim_\N$ are pre-additions, we can form the sequences
 $$  \sum a_i+\sum \tilde a_{\tilde i}\ \=\ \sum c_{1,k}+\sum \tilde c_{1,\tilde k} \ \sim_\N \  \dotsb \quad \sim_\N \ \sum d_{n,k}+\sum \tilde d_{n,\tilde k} \ \= \ \sum b_j+\sum \tilde b_{\tilde j} $$
 and 
 $$  \sum a_i\tilde a_{\tilde i}\ \=\ \sum c_{1,k}\tilde c_{1,\tilde k} \ \sim_\N \  \dotsb \quad \sim_\N \ \sum d_{n,k}\tilde d_{n,\tilde k} \ \= \ \sum b_j\tilde b_{\tilde j}, $$
 which shows the additivity and multiplicativity of $\cR_\sim'$. 

 We proceed to show that $B/\sim$ is proper and that $\sim$ is the kernel of $p:B\to B/\sim$, i.e.\ we show that $p(a)\=p(b)$ implies $p(a)=p(b)$ and $p(a)=p(b)$ implies $a\sim b$, respectively. Since, conversely, $p(a)\=p(b)$ follows from $p(a)=p(b)$, and $p(a)=p(b)$ follows from $a\sim b$, everything is proven if we show that $a\sim b$ follows from $p(a)\=p(b)$. By what we have shown above, $p(a)\=p(b)$ implies that there is a sequence
 $$ a\ \=\ \sum c_{1,k} \ \sim_\N \quad \dotsb \quad \sim_\N \ \sum d_{n,k} \ \= \ b, $$
 and property \textup{(C2)}$^*$ of a congruence implies in turn that $a\sim b$, which completes the proof.
\end{proof}

We prove some further properties of congruences.

\begin{lemma}\label{lemma:easy_properties_of_congruences}
 If $\sim$ is a congruence on $B$, then it satisfies the following properties.
 \begin{enumerate}
  \item If $a\= b$, then $a\sim b$.
  \item If $a+\sum c_k\=b+\sum d_l$ and $c_k\sim e\sim d_l$ where $e\=\zero$ is a zero of $B$, then $a\sim b$.
 \end{enumerate}
\end{lemma}

\begin{proof}
 The first statement follows immediately from \textup{(C2)}$^*$. The second statement follows from \textup{(C2)}$^*$ applied to the sequence
 \[ 
   a \ \= \ a +\sum e \ \sim_\N \ a+\sum c_k \ \= \ b+\sum d_l \ \sim_\N b+\sum e \ \= \ b. \qedhere 
 \]
\end{proof}

\begin{ex}
 Every blueprint has a maximal and a minimal congruence. The maximal congruence is the relation defined by $a\sim b$ for all $a,b\in B$. The relation $\sim$ is indeed a congruence since it is the kernel of the morphism $B\to \{1\}$ into the trivial blueprint with $1\=\zero$ (cf.\ Proposition \ref{prop:kernels_are_congruences}).

 The minimal congruence is the equivalence relation on $B$ defined by $a\sim b$ if and only if $a\= b$. We described the quotient $B/\sim$ in Lemma \ref{lemma_proper}, which is nothing else that $B_\proper$. Since $\sim$ is the kernel of the quotient map $B\to B/\sim$, it follows a posteriori that $\sim$ is a congruence. By Lemma \ref{lemma:easy_properties_of_congruences}, $\sim$ is contained in any other congruence on $B$.
\end{ex}

%%%%%%%%%%%%%%%%%%%%%%%%%%%%%%%%%%%%%%%%%%%%%%%%%%%%%%%%%%%%%%%%%%%%%%%%%%%%%%%%%%%%%%%%%%%%%%%%%%%%%%%%%%%%%%%%%%%%%%%%%%%%%%%%%%%%%%%%%%%%%%%%%%%%%%%%%%%%%%%%%%%%%%%%%%%%%%%%%%%%%%%%%%%%%%%%%%%%%%%%%%%%%%%%%%%%%%%%%%%%

\subsection{Ideals}
\label{subsection:ideals}

 We introduce ideals of blueprints, led by the idea that ideals should coincide with the inverse images of $0$ under morphisms. Since not every blueprint has a zero, we have find a generalization of this definition. This will be motivated by the notion of an absorbing ideal of a congruence as defined below.

\begin{df}
 Let $\sim$ be a congruence on $B$. The \emph{absorbing ideal of $\sim$} is the set 
 $$ I_\sim \ = \ \{\ e\in B\ |\ eb\sim e\text{ for all }b\in B\ \}. $$
 Let $f:B\to C$ be a morphism of blueprints. The \emph{absorbing ideal of $f$} is $I_f=I_{\sim_f}$ where $\sim_f$ is the kernel of $f$.
\end{df}

We collect some immediate properties of $I_\sim$.

\begin{lemma}\label{lemma:absorbing_ideal_is_an_equivalence_class}
 The absorbing ideal $I_\sim$ is either empty or an equivalence class of $\sim$.
\end{lemma}

\begin{proof}
 If $e,e'\in I_\sim$, then $e\sim ee'\sim e'$ and, conversely, if $e\in I_\sim$ and $e\sim e'$, then $e'b\sim eb\sim e\sim e'$ for all $b\in B$, which means that $e'\in I_\sim$. Thus $I_\sim$ equals an equivalence class of $\sim$ if it is not empty.
\end{proof}

An \emph{absorbing element of $B$} is an element $e\in B$ with the property that $eb\=e$ for all $b\in B$. If $e$ and $e'$ are absorbing elements, then $e'\=ee'\=e$. If $e\=\zero$ is a zero of $B$, then $eb\=\zero\=e$, which implies that $e$ is an absorbing element. 

\begin{lemma}\label{lemma:absorbing_elements_in_the_absorbing_ideal}
 If $B$ has an absorbing element $e$, then $I_\sim=p^{-1}(p(e))$ where $p:B\to B/\sim$ is the canonical projection.
\end{lemma}

\begin{proof}
 By Lemma \ref{lemma:absorbing_ideal_is_an_equivalence_class}, $I_\sim$ equals the inverse image of precisely one element $[a]$ in $B/\sim$ if $I_\sim$ is not empty. Since $eb\=e$ for all $b\in B$ by the definition of an absorbing element and since $eb\=e$ implies $eb\sim e$ by Lemma \ref{lemma:easy_properties_of_congruences} \textup{(C1)}, it follows that $e\in I_\sim$. Thus $[e]=p(e)$ is the class whose inverse image is $I_\sim$.
\end{proof}

Let $I\subset B$ be a subset. Let $\sim^I$ be the equivalence relation on $A$ that is defined by $a\sim^I b$ if and only if $a=b$ or $a,b\in I$, and $\sim_I$ be the restriction of $\sim^I_\cR$ to $A$ (cf.\ Definition \ref{def:congruence}). Recall from the proof of Lemma \ref{lemma:equivalent_definition_of_a_congruence} that $a\sim_I b$ if and only if there exists a sequence
$$ a \ \= \ \sum c_{1,k} \ \sim^I_\N \ \sum d_{1,k} \ \= \ \sum c_{2,k} \ \sim^I_\N \quad \dotsb \quad \sim^I_\N \ \sum d_{n,k} \ \= \ b $$
where $\sim^I_\N$ is the linear extension of $\sim^I$ to $\N[A]$, i.e.\ $\sum c_k\sim^I_\N\sum d_k$ if $c_{k}\sim^I d_{k}$ for all $k$.

\begin{df}
 An \emph{ideal of $B$} is a subset $I$ of $B$ that satisfies the following properties.
 \begin{enumerate}
  \item[\textup{(I1)}] If $a\in I$ and $b\in B$, then $ab\in I$.
  \item[\textup{(I2)}] If $e$ is a zero of $B$, then $e\in I$.
  \item[\textup{(I3)}] If $a\sim_I b$ and $b\in I$, then $a\in I$.
 \end{enumerate}
 If $I$ is an ideal, then $\sim_I$ is called the \emph{congruence generated by $I$}.
\end{df}

Note that if $B$ does not have a zero, then the empty set is an ideal of $B$. If $B$, however, has an absorbing element, the empty set cannot be an absorbing ideal. This digression is unavoidable if we want inverse images of ideals to be ideals (cf.\ Proposition \ref{prop_inverse_images}). The following two propositions show that apart from this case, ideals and absorbing ideals coincide. In particular, in $\bp_0$, there is a complete analogy between ideals and absorbing ideals. Note that if $I$ is non-empty, then $I$ contains the absorbing elements of $B$. 

\begin{prop}\label{prop:congruence_of_an_ideal}
 Let $I$ be an ideal that contains all absorbing elements of $B$. Then $\sim_I$ is the smallest congruence whose absorbing ideal is $I$.
\end{prop}

\begin{proof}
 Once we have proven that $\sim_I$ is a congruence with vanishing ideal $I$, it is clear that it is the smallest congruence with this property by its definition.

 We prove that $\sim_I$ is a congruence. First note that $\sim^I$ is multiplicative: given $a\sim^I b$ and $c\sim^I d$, then either $a=b$ and $c=d$ or $ac,bd\in I$. In both cases it follows that $ac\sim^I bd$. By Lemma \ref{lemma:equivalent_definition_of_a_congruence} \textup{(C1)}$^*$, this implies that $\sim^I_\N$ is a pre-addition. Thus we can use the same argument as in the proof of Proposition \ref{prop:congruences_are_kernels} to show that $\sim_I$ satisfies property \textup{(C1)}$^*$.

 We proceed to show that $\sim_I$ satisfies property \textup{(C2)}$^*$. Consider a sequence
 $$ a \ \= \ \sum c_{1,k} \ \sim_{I,\N} \ \sum d_{1,k} \ \= \ \sum c_{2,k} \ \sim_{I,\N} \quad \dotsb \quad \sim_{I,\N} \ \sum d_{n,k} \ \= \ b. $$
 Unraveling the definition of $\sim_{I,\N}$, we see that $\sum c_k\sim_{I,\N}\sum d_k$ if and only if for every $k$, there is a sequence
 $$ c_k \ \= \sum e_{1,l} \ \sim^I_\N \quad \dotsb \quad \sim^I_\N \ \sum f_{m,l} \ = \ d_k. $$
 Adding up these sequences (both $\=$ and $\sim^I_\N$ are additive) and inserting them into the sequence for $a$ and $b$ yields a sequence of the form
 $$ a \ \= \ \sum \tilde c_{1,\tilde k} \ \sim^I_{\N} \ \sum \tilde d_{1,\tilde k} \ \= \ \sum \tilde c_{2,\tilde k} \ \sim^I_{\N} \quad \dotsb \quad \sim^I_{\N} \ \sum \tilde d_{\tilde n,\tilde k} \ \= \ b. $$
 Thus $a\sim_I b$, which completes the proof that $\sim_I$ is a congruence.

 Let $I_{\sim_I}$ be the absorbing ideal of $\sim_I$. We show that $I_{\sim_I}=I$. Let $a\in I$, then $ab\in I$ for any $b\in B$, or in other words, $ab\sim^I a$. This implies $ab\sim_I a$. Thus $a\in I_{\sim_I}$, and we have shown that $I\subset I_{\sim_I}$. Let, conversely, $a\in I_{\sim_I}$, i.e.\ $ab\sim_I a$ for all $b\in B$. Assume that $I$ contains an element $e$. Then $a\sim_I ae$, where $ae\in I$ by axiom \textup{(I1)}. By axiom \textup{(I3)}, also $a\in I$.

 We have to exclude the case that $I$ is empty while $I_{\sim_I}$ is not. If $I$ is empty, then both $\sim^I$ and $\sim_I$ equal $\=$. Therefore $I_{\sim_I}$ equals the set of the absorbing elements of $B$. By the hypothesis of the proposition, $I$ contains all absorbing elements, which implies that $I_{\sim_I}$ is empty. This completes the proof of the proposition.
\end{proof}

\begin{prop}\label{prop:vanishing_ideals_are_ideals}
 The absorbing ideal of a congruence is an ideal.
\end{prop}

\begin{proof}
 Let $\sim$ be a congruence on $B$ and $I_\sim$ its absorbing ideal. We verify the axioms of an ideal. Let $a\in I_\sim$ and $b\in B$, then we have for all $c\in B$ that $(ab)c\sim (ac)b\sim ab$ and thus $ab\in I_\sim$. This shows axiom \textup{(I1)} of an ideal.

 Let $a$ be a zero of $B$. Then $ab\=a$ for all $b\in B$. This implies $ab\sim b$ for all $b\in B$ and thus $a\in I_\sim$. This shows axiom \textup{(I2)} of an ideal.

 Let $\sim_I$ be the congruence generated by $I=I_\sim$. Let $a\sim_I b$ and $b\in I_\sim$. We have to show that $a\in I_\sim$. Since $a\sim_I b$, there is a sequence
 $$ a \ \= \ \sum c_{1,k} \ \sim_{I,\N} \quad \dotsb \quad \sim_{I,\N} \ \sum d_{n,k} \ \= \ b, $$
 which implies
 $$ a \ \= \ \sum c_{1,k} \ \sim_{\N} \quad \dotsb \quad \sim_{\N} \ \sum d_{n,k} \ \= \ b $$
 since $\sim_I$ is the smallest congruence with absorbing ideal $I_\sim$ by Proposition \ref{prop:congruence_of_an_ideal}. Let $e\in B$, then the multiplicativity of $\=$ and $\sim$ yields
 $$ ae \ \= \ \sum c_{1,k}e \ \sim_{\N} \quad \dotsb \quad \sim_{\N} \ \sum d_{n,k}e \ \= \ be, $$
 which means that $ae\sim be$ by property \textup{(C2)}$^*$ of a congruence. Thus we have that $ae\sim be\sim b\sim a$ for every $e\in B$ and consequently $a\in I_\sim$. This shows axiom \textup{(I3)} of an ideal.
\end{proof}

Define the quotient $B/I$ as $B/\sim_I$ for an ideal $I$ of $B$. The following statements follow immediately from the above propositions.

\begin{cor}
 Let $I$ be an ideal that contains all absorbing elements of $B$. Then $I$ equals $p^{-1}(e)$ if $p:B\to B/I$ is the quotient map and $e$ is an absorbing element of $B/I$. If $B/I$ does not have an absorbing element, $I$ is empty and $p$ is an isomorphism. More generally, if $f:B\to C$ is a morphism of blueprints and $C$ has an absorbing element $e$, then $I_f=f^{-1}(e)$. \qed
\end{cor}

\begin{pg}\label{pg:generated_ideal_and_congruence}
 For a subset $J$ of $B$, we define the equivalence relation $\sim^{J}$ on $B$ by $a\sim^Jb$ if and only if either $a\=b$ or $a,b\in JB=\{c\in B|c\=md\text{ with }m\in J, d\in B\}$. Let $\sim^J_\N$ be the linear extension to $\N[A]$. The \emph{ideal generated by $J$} is the set
 \[
  \igen{J}_B \ = \ \{ \ b\in B \ | \ b\=\sum c_{1,k}\sim^{J}_\N \quad\dotsb\quad \sim^{J}_\N \sum d_{n,k}\= a\text{ with }a\in JB\text{ and }c_{i,k},d_{i,k}\in B \ \}.
 \]
 For a set $J=\{a_i\sim b_i\}_{i\in I}$, we define the equivalence relation $\sim^J$ on $B$ by $c\sim^J d$ if and only if $c\=d$ or $c=fa_i$ and $d=fb_i$ resp.\ $c=fb_i$ and $d=fa_i$ for some $i\in I$ and $f\in B$. The \emph{congruence generated by $J$} is the relation $\sim_J$ on $B$ with $a\sim_J b$ if and only if there is a sequence $a\=\sum c_{1,k}\sim^{J}_\N \quad\dotsb\quad \sim^{J}_\N \sum d_{n,k}\= b$ for some $c_{i,k},d_{i,k}\in B$.

 We leave it to the reader to verify that $\igen J_B$ is the smallest ideal of $B$ that contains $J$ and that $\sim_J$ is the smallest congruence that contains $J$.
\end{pg}

%%%%%%%%%%%%%%%%%%%%%%%%%%%%%%%%%%%%%%%%%%%%%%%%%%%%%%%%%%%%%%%%%%%%%%%%%%%%%%%%%%%%%%%%%%%%%%%%%%%%%%%%%%%%%%%%%%%%%%%%%%%%%%%%%%%%%%%%%%%%%%%%%%%%%%%%%%%%%%%%%%%%%%%%%%%%%%%%%%%%%%%%%%%%%%%%%%%%%%%%%%%%%%%%%%%%%%%%%%%%
\subsection{Ideals for blueprints with a zero}
\label{section:ideals_with_zero}

Axiom \textup{(I3)} of an ideal can be reformulated in a more convenient way if the blueprint has a zero. 

\begin{lemma}\label{lemma_ideals_in_bp_with_0}
 Let $B$ be a blueprint with a zero $0$. Then a subset $I\subset B$ is an ideal if and only if it satisfies the following three properties.
 \begin{enumerate}
  \item\label{bp0ideal1} $IB\subset I$.
  \item\label{bp0ideal2} $0\in B$.
  \item\label{bp0ideal3} If $a+\sum b_j\=\sum c_k$ and $b_j,c_k\in I$, then $a\in I$.
 \end{enumerate}
\end{lemma}

\begin{proof}
 The first two properties coincide with the first two axioms of an ideal. We show that property \eqref{bp0ideal3} implies axiom \textup{(I3)}. Consider a sequence 
 $$ a \ \= \ \sum c_{1,k} \ \sim^I_\N \ \sum d_{1,k} \ \= \ \sum c_{2,k} \ \sim^I_\N \quad \dotsb \quad \sim^I_\N \ \sum d_{n,k} \ \= \ b $$
 with $b\in I$. Then we have for every $i\in\{1,\dotsc,n\}$ that $\sum c_{i,k}= \sum f_{i,j} + \sum \tilde c_{i,l}$ and $\sum d_{i,k}= \sum f_{i,j} + \sum \tilde d_{i,l}$ for certain $f_{i,j}\in B$ and $\tilde c_{i,l},\tilde d_{i,l}\in I$ by the definition of $\sim^I_\N$. Rewriting the above sequence yields 
 \begin{multline*}
 a \quad \= \quad \sum f_{1,j} + \sum \tilde c_{1,l} \quad \sim^I_\N \quad \sum f_{1,j} + \sum \tilde d_{1,l} \quad \= \quad \sum f_{2,j} + \sum \tilde c_{2,l} \quad \sim^I_\N \quad \dotsb \\
                \dotsb \quad \sim^I_\N \quad \sum f_{n,j} + \sum \tilde d_{n,l} \quad \= \quad b.
 \end{multline*}
 This gives rise to a sequence of additive relations
 \begin{align*}
  a + \sum \tilde d_{1,l} + \dotsb + \sum \tilde d_{n,l} \quad &\= \quad \sum f_{1,j} + \sum \tilde c_{1,l} + \sum \tilde d_{1,l} + \dotsb + \sum \tilde d_{n,l} \\
                                                               &\= \quad \sum f_{2,j} + \sum \tilde c_{1,l} + \sum \tilde c_{2,l} + \sum \tilde d_{2,l} + \dotsb + \sum \tilde d_{n,l} \\
                                                               &\ \ \vdots  \\
                                                               &\= \quad \sum f_{n,j} + \sum \tilde c_{1,l} + \dotsb + \sum \tilde c_{n,l} + \sum \tilde d_{n,l} \\
                                                               &\= \quad \sum b + \sum \tilde c_{1,l} + \dotsb + \sum \tilde c_{n,l}. \\
 \end{align*}
 We can apply property \eqref{bp0ideal3} to the first and last term in this sequence to conclude that $a\in I$. This establishes axiom \textup{(I3)}. 

 To derive the converse implication, consider an additive relation $a+\sum b_j\=\sum c_k$ with $b_j,c_k\in I$. Then axiom \textup{(I3)} applied to the sequence
 $$ a \ \= \ a + \sum 0 \ \sim^I_\N \ a+\sum b_j \ \= \ \sum c_k \ \sim^I_\N \ \sum 0 \ \= \ 0 $$
 yields $a\in I$. This establishes property \eqref{bp0ideal3}.
\end{proof}

%%%%%%%%%%%%%%%%%%%%%%%%%%%%%%%%%%%%%%%%%%%%%%%%%%%%%%%%%%%%%%%%%%%%%%%%%%%%%%%%%%%%%%%%%%%%%%%%%%%%%%%%%%%%%%%%%%%%%%%%%%%%%%%%%%%%%%%%%%%%%%%%%%%%%%%%%%%%%%%%%%%%%%%%%%%%%%%%%%%%%%%%%%%%%%%%%%%%%%%%%%%%%%%%%%%%%%%%%%%%
\subsection{Congruences and ideals for monoids}
\label{section:congruences_and_ideals_for_monoids}

 Let $B=\bpquot A\cR$ be a monoid, i.e.\ $\cR=\{\sum a_i\=\sum a_i| a_i\in A\}$ is the trivial pre-addition. Then $I$ is an ideal of $B$ if and only if $I$ satisfies axiom \textup{(I1)}, i.e.\ if $IB\subset I$, since axioms \textup{(I2)} and \textup{(I3)} are automatically satisfied in the case of monoids. This is the same definition as in Deitmar's approach to $\Fun$-schemes (cf.\ \cite{Deitmar05}). Since property \textup{(C2)}$^*$ of a congruence is automatically satisfied by an equivalence relation, a congruence is the same as what is usually known as a congruence for semi-groups: a multiplicative equivalence relation (cf.\ \cite{Clifford-Preston1}).

 Let $B=\bpquot A\cR$ be a monoid with a zero, i.e.\ $\cR=\gen{0\=\zero}$. Then $I$ is an ideal of $B$ if and only if $IB\subset I$ and $0\in I$ since also in this case, axiom \textup{(I3)} is automatically satisfied as can be seen easily from Lemma \ref{lemma_ideals_in_bp_with_0}. Therefore our notion of an ideal coincides with the one introduced by Connes and Consani for monoids with a zero (cf.\ \cite{CC09}). Also in the case of monoids with a zero, property $\textup{(C2)}^*$ of a congruence is automatically satisfied and the notion of a congruence reduces to that of a multiplicative equivalence relation.

%%%%%%%%%%%%%%%%%%%%%%%%%%%%%%%%%%%%%%%%%%%%%%%%%%%%%%%%%%%%%%%%%%%%%%%%%%%%%%%%%%%%%%%%%%%%%%%%%%%%%%%%%%%%%%%%%%%%%%%%%%%%%%%%%%%%%%%%%%%%%%%%%%%%%%%%%%%%%%%%%%%%%%%%%%%%%%%%%%%%%%%%%%%%%%%%%%%%%%%%%%%%%%%%%%%%%%%%%%%%
\subsection{Congruences and ideals for rings}
\label{section:congruences_and_ideals_for_rings}

 If $B$ is a ring, then it is clear by Lemma \ref{lemma_ideals_in_bp_with_0} that the definition of an ideal defined in this text coincide with the usual definition of an ideal. Further, it is well-known that every quotient of a ring can be described as the quotient structure on the set of coset with respect to an ideal of $R$. Thus every congruence $\sim$ on a ring $B$ is of the form $\sim_I$ for some ideal $I$ of $B$. We have the following compatibility with the functor $(\blanc)_\Z:\bp\to\Rings$.

\begin{lemma}\label{lemma:base_extension_of_congruences_and_ideals_to_z}
  Let $B=\bpquot A\cR$ be a blueprint and $\sim$ a congruence on $B$. Let $i:B\to B_\Z$ be the canonical map. Then 
 \[ 
  I_\Z(\sim) \ = \ \{\ a\in B_\Z \ | \ a=\sum i(c_k)-\sum i(d_k)\text{ for }c_k, d_k \in B\text{ with }c_k\sim d_k \ \} 
 \] 
 is an ideal of $B_\Z$ and $(B/\sim)_\Z\simeq B_\Z/I_\Z(\sim)$. If $\sim=\sim_I$ for an ideal $I$ of $B$, then $I_\Z(\sim_I)$ equals
 \[
  I_\Z \ = \ \{\ a\in B_\Z \ | \ a=\sum i(c_k)-\sum i(d_k)\text{ with }c_k, d_k \in I \ \}.
 \]

\end{lemma}

\begin{proof}
 We verify that $I_\Z(\sim)$ is an ideal. It contains the zero $0$ of $B_\Z$ since $0=i(1)-i(1)$. If $a,a'\in I_\Z(\sim)$, i.e.\ $a=\sum i(c_k)-\sum i(d_k)$ and $a'=\sum i(c_k')-\sum i(d_k')$ with $c_k\sim d_k$ and $c_k'\sim d_k'$, then $a-a'=\sum i(c_k)+\sum i(d_k')-\sum i(d_k)-\sum i(c_k')$ is an element of $I_\Z(\sim)$.

 The equality $(B/\sim)_\Z\simeq B_\Z/I_\Z(\sim)$ follows from the observation that both rings in question are quotients of $\Z[A]$ by the ideal generated by the elements of the form $\sum a_i-\sum b_j$ where $\sum a_i\= \sum b_j$ and $a-b$ where $a\sim b$.

 If $\sim=\sim_I$ for an ideal $I$ of $B$, then $c_k\sim d_k$ if and only if $i(c_k)$ and $i(d_k)$ represent the same coset of $I_\Z$, i.e.\ $i(c_k)+I_\Z=i(d_k)+I_\Z$. Therefore $I_\Z(\sim_I)$ is exhausted by differences $\sum i(c_k)-\sum i(d_k)$ with $c_k,d_k\in I$.
\end{proof}

%%%%%%%%%%%%%%%%%%%%%%%%%%%%%%%%%%%%%%%%%%%%%%%%%%%%%%%%%%%%%%%%%%%%%%%%%%%%%%%%%%%%%%%%%%%%%%%%%%%%%%%%%%%%%%%%%%%%%%%%%%%%%%%%%%%%%%%%%%%%%%%%%%%%%%%%%%%%%%%%%%%%%%%%%%%%%%%%%%%%%%%%%%%%%%%%%%%%%%%%%%%%%%%%%%%%%%%%%%%%
\subsection{Congruences and ideals for semirings}
\label{section:congruences_and_ideals_for_semirings}
 
Let $B$ be a semiring and $0$ its zero. Since $B$ is proper, we write $=$ in place of $\=$, and since $B$ is a semiring, we identify $a+b$ with $c$ if $c$ is the unique element such that $a+b=c$. The conditions stated in Lemma \ref{lemma_ideals_in_bp_with_0} coincide with the axioms of what is called a $k$-ideal in semiring theory---a semiring ideal in the usual sense is a subset $I$ of $B$ such that $IB\subset I$, $0\in B$ and such that $a+b\in I$ whenever $a,b\in I$ (cf.\ \cite{LaTorre65}). For the purpose of scheme theory, however, the notion of a $k$-ideal is the appropriate one. 

We can reformulate the axioms of a congruence in an easier shape.

\begin{lemma}
 Let $B$ be a semiring. Then a relation $\sim$ on $B$ is a congruence if and only if for all $a,b,c,d\in B$ with $a\sim b$ and $c\sim d$, we have $ac\sim bd$ and $a+c\sim b+d$.
\end{lemma}

\begin{proof}
 The lemma is proven if we have shown that the property that $a\sim b$ and $c\sim d$ implies $a+c\sim b+d$ is equivalent to property \textup{(C2)}$^*$ of a congruence. Assume $\sim$ is a congruence and let $e=a+c$ and $f=b+d$. Then the sequence $e=a+c\sim_\N b+d=f$ shows that $e\sim f$. 

 Conversely, assume $\sim$ satisfies the properties of the lemma and consider a sequence
 $$ a \ = \ \sum c_{1,k} \ \sim_\N \ \sum d_{1,k} \ = \ \sum c_{2,k} \ \sim_\N \quad \dotsb \quad \sim_\N \ \sum d_{n,k} \ = \ b. $$
 Since for $i=1,\dotsc,n-1$, there is a $b_i\in B$ such that $\sum d_{1,k} = b_1 = \sum c_{2,k}$, we can break the above sequence into pieces of the form 
 \[
  a \ = \ \sum c_{1,k} \ \sim_\N \ \sum d_{1,k} \ = \ b_1,  \qquad  b_1 \ = \ \sum c_{2,k} \ \sim_\N \ \sum d_{2,k} \ = \ b_2 \quad\dotsc \ ,
 \]
 which reduces the proof to show that $a = \sum c_{k} \sim_\N \sum d_{k}= b$ implies $a\sim b$. But this follows easily by an induction on the number of summands in $\sum c_k$ from the properties of the lemma.
\end{proof}

%%%%%%%%%%%%%%%%%%%%%%%%%%%%%%%%%%%%%%%%%%%%%%%%%%%%%%%%%%%%%%%%%%%%%%%%%%%%%%%%%%%%%%%%%%%%%%%%%%%%%%%%%%%%%%%%%%%%%%%%%%%%%%%%%%%%%%%%%%%%%%%%%%%%%%%%%%%%%%%%%%%%%%%%%%%%%%%%%%%%%%%%%%%%%%%%%%%%%%%%%%%%%%%%%%%%%%%%%%%%

\subsection{Prime congruences and prime ideals}
\label{section:prime_congruences_and_prime_ideals}

In this section, we introduce prime ideals and prime congruences. Let $B=\bpquot A\cR$ be a blueprint throughout this section. 

An ideal $I$ of $B$ is \emph{proper} if $I\neq B$. A congruence $\sim$ on $B$ is proper if it has more than one equivalence classes.

\begin{lemma}\label{lemma:proper_congruence_and_proper_ideal}
 A congruence $\sim$ is proper if and only if its absorbing ideal $I_\sim$ is proper.
\end{lemma}

\begin{proof}
 Since $I_\sim$ equals an equivalence class of $\sim$ if it is not empty, everything is clear once we showed that $I_\sim$ is not empty if $\sim$ is proper. Assume $a\sim b$ for all $a,b\in B$. Then also $ab\sim a$ for all $a,b\in B$ and thus $a\in I_\sim$.
\end{proof}

\begin{df}\label{def_prime_ideal_and_prime_congruence} \ 
 \begin{enumerate}
  \item A \emph{prime ideal} is an ideal $\fp$ of $B$ such that the complement $B-\fp$ is a multiplicative subset. 
  \item A \emph{prime congruence} is a proper congruence $\sim$ on $B$ such that the quotient $B/\sim$ is an integral blueprint.
 \end{enumerate}
\end{df}

The following reformulations of the definition of prime ideals and prime congruences are immediate.

\begin{lemma}\label{lemma_prime_ideal_and_prime_congruence}\ 
 \begin{enumerate}
  \item An ideal $I$ is prime if and only if for all $a,b\in B$ such that $ab\in I$, either $a\in I$ or $b\in I$. 
  \item A proper congruence $\sim$ is prime if and only if for all $a,b,c\in B$ such that $ab\sim ac$, either $b\sim c$ or $a\sim e$ where $e\=\zero$ is a zero.\qed
 \end{enumerate}
\end{lemma}

\begin{lemma}\label{lemma:_absorbing_ideal_of_a_prime_congruence_is_prime}
 If $\sim$ is a prime congruence, then its absorbing ideal $I_\sim$ is a prime ideal.
\end{lemma}

\begin{proof}
 Let $ab\in I_\sim$. Then $ab\sim a(ab)$ by the definition of the absorbing ideal. Since $\sim$ is prime, either $b\sim ab$ or $a\sim e$ with $e\=\zero$. Since both $ab$ and $e$, if it exists, are elements in $I_\sim$, either $a$ or $b$ is so, too. This shows that $B-I_\sim$ is a multiplicative subset and thus that $I_\sim$ is prime.
\end{proof}

%\begin{pg}
 Consequently, an ideal $I$ is prime if the quotient $B/I$ is an integral blueprint. More precisely, an ideal $I$ is prime if and only if the quotient $B/I$ is without (non-trivial) zero-divisors. An integral blueprint is clearly without zero divisors, but the contrary is not true, in contrast to the situation for rings: the monoid $B=\{0,e,1\}$ with zero $0\=\zero$ and idempotent element $e^2=e$ is without zero divisors, but not integral. Moreover, $I=\{0\}$ is a prime ideal, but $\sim_I$ is not a prime congruence.

 Let $f:B\to C$ be a morphism of blueprints and $\sim$ a congruence on $C$. The inverse image of $\sim$ is the relation $\sim^*=f^{-1}(\sim)$ on $B$ that is defined by $a\sim^* b$ if and only if $f(a)\sim f(b)$.
%\end{pg}

\begin{prop}\label{prop_inverse_images}
 Let $f:B\to C$ be a morphism of blueprints. 
 \begin{enumerate}
  \item\label{inverse1} Let $I$ be an ideal of $C$. Then $f^{-1}(I)$ is an ideal of $B$. If $I$ is prime, then $f^{-1}(I)$ is prime.
  \item\label{inverse2} Let $\sim$ be a congruence on $C$. Then $f^{-1}(\sim)$ is a congruence of $B$. If $\sim$ is prime and if either $B$ is with a zero or $C$ is without a zero, then $f^{-1}(\sim)$ is prime.
  \item\label{inverse3} Let $\sim$ be a congruence on $C$ and $I_\sim$ its absorbing ideal. Assume that both $B$ and $C$ have a zero. Then $f^{-1}(I_\sim)$ equals the absorbing ideal of $f^{-1}(\sim)$.
 \end{enumerate}
\end{prop}

We need to make an additional assumption on the existence of zeros in parts \eqref{part2} and \eqref{part3} of the proposition as the following examples show. Concerning \eqref{inverse2}, let $A=\{0,1\}$ and define $B=\bpgenquot{A}{\emptyset}$ and $C=\bpgenquot{A}{0\=\zero}$. The identity morphism on $A$ induces a morphism $f:B\to C$ of blueprints. Then the minimal congruence $\sim$ on $C$ with two equivalence classes $\{0\}$ and $\{1\}$ is a prime congruence while $f^{-1}(\sim)$ is not. Concerning \eqref{part3}, consider the inclusion $\{1\}\to\{0,1\}$ of monoids and let $\sim$ be the minimal congruence on $C$. Then the vanishing ideal of $\sim$ is $\{0\}$ while the vanishing ideal of its inverse image is $\{1\}$.

\begin{proof}
 We prove \eqref{part1} of the proposition. We show that $J=f^{-1}(I)$ satisfies the three axioms of an ideal. Let $a\in J$ and $b\in B$, then $f(ab)=f(a)f(b)\in I$ and thus $ab\in J$. This shows axiom \textup{(I1)}. If $e\=\zero$ is a zero of $B$, then $f(e)\=\zero$ is  a zero of $C$ and thus contained in $I$, which means that $e\in J$. This shows axiom \textup{(I2)}. Given a sequence
 \[
   a \ = \ \sum c_{1,k} \ \sim^J_\N \quad \dotsb \quad \sim^J_\N \ \sum d_{n,k} \ = \ b  
 \]
 in $B$ with $b\in J$, then we have 
 \[
   f(a) \ = \ \sum f(c_{1,k}) \ \sim^I_\N \quad \dotsb \quad \sim^I_\N \ \sum f(d_{n,k}) \ = \ f(b)  
 \]
 in $C$ with $f(b)\in I$. Therefore $f(a)\in I$ and $a\in J$. This shows axiom \textup{(I3)}.

 Suppose that $I$ is prime. If $ab\in J$, then $f(ab)\in I$ and either $f(a)\in I$ or $f(b)\in I$. This means that either $a\in J$ or $b\in J$. Thus $J$ is prime.

 We proceed with \eqref{part2}. Let $\sim$ be a congruence on $C$ and let $\sim^*=f^{-1}(\sim)$ be its inverse image. We verify the two axioms of a congruence. If $a\sim^*b$ and $c\sim^*d$, then $f(a)\sim^*f(b)$ and $f(c)\sim^*f(d)$. Therefore $f(ac)=f(a)f(c)\sim f(b)f(d)=f(bd)$ and thus $ac\sim^* bd$. This shows property \textup{(C1)}$^*$ of a congruence. Given a sequence
 \[
   a \ = \ \sum c_{1,k} \ \sim^*_\N \quad \dotsb \quad \sim^*_\N \ \sum d_{n,k} \ = \ b  
 \]
 in $B$, then we have 
 \[
   f(a) \ = \ \sum f(c_{1,k}) \ \sim_\N \quad \dotsb \quad \sim_\N \ \sum f(d_{n,k}) \ = \ f(b)  
 \]
 in $C$. Therefore $f(a)\sim f(b)$, which means that $a\sim^* b$. This shows property \textup{(C2)}$^*$ of a congruence.

 Suppose that $\sim$ is prime and that $B$ is with a zero or $C$ is without a zero. If $ab\sim^*ac$, then $f(a)f(b)=f(ab)\sim f(ac)=f(a)f(c)$. Therefore either $f(b)\sim f(c)$ or $f(a)\sim e$ where $e$ is a zero. In the former case, we have $b\sim^* c$, in the latter case, $B$ has a zero $e'$ by assumption, which maps to $f(e')\=\zero\=e$; this means that $f(a)\sim f(e')$ and $a\sim^* e'$. Thus $\sim^*$ is prime.

 We prove \eqref{part3}. Let $\sim^*=f^{-1}(\sim)$ be the inverse image of $\sim$. Then the absorbing ideal $I_{\sim^*}$ of $\sim^*$ equals the equivalence class of a zero $e\=\zero$. On the other hand, it is clear that the inverse image of an equivalence class of $\sim$ is an equivalence class of $\sim^*$. Since $f(e)$ is a zero in $C$ and therefore contained in $\in I_\sim$, we have that $f^{-1}(I_\sim)=I_{\sim^*}$. This completes the proof of the proposition.
\end{proof}

%%%%%%%%%%%%%%%%%%%%%%%%%%%%%%%%%%%%%%%%%%%%%%%%%%%%%%%%%%%%%%%%%%%%%%%%%%%%%%%%%%%%%%%%%%%%%%%%%%%%%%%%%%%%%%%%%%%%%%%%%%%%%%%%%%%%%%%%%%%%%%%%%%%%%%%%%%%%%%%%%%%%%%%%%%%%%%%%%%%%%%%%%%%%%%%%%%%%%%%%%%%%%%%%%%%%%%%%%%%%

\subsection{Maximal congruences and maximal ideals}
\label{section:maximal_congruences_and_maximal_ideals}

A \emph{maximal congruence} is a proper congruence that is not contained in any larger proper congruence.  A \emph{maximal ideal} is a proper ideal that is not contained in any larger proper ideal. 

\begin{prop}
 A maximal ideal is a prime ideal. If $B$ is with a zero, then a maximal congruence is a prime congruence.
\end{prop}

\begin{proof}
 Let $\fm$ be a maximal ideal and assume $ab\in\fm$, but $a\notin\fm$. We have to show that $b\in\fm$. Let $J=\fm\cup\{a\}$, then $\igen J_B=B$ by the maximality of $\fm$. Thus there is a sequence 
 \[
  1 \ \= \ \sum c_{1,k} \ \sim^{J}_\N \quad\dotsb\quad \sim^{J}_\N  \ \sum d_{n,k} \ \= \ f
 \]
 where $f\in JB=\fm\cup aB$. If $c\sim^J d$, then $bc\sim^\fm bd$ since $ab\in\fm$. Thus multiplying the above sequence by $b$ yields
 \[
  b \ \= \ \sum bc_{1,k} \ \sim^\fm_\N \quad\dotsb\quad \sim^\fm_\N  \ \sum bd_{n,k} \ \= \ bf
 \]
 where $bf\in\fm$. This means that $b\in \fm$ what we had to show. Thus $\fm$ is prime.

 Let $e\=\zero$ be a zero of $B$ and $\sim$ a maximal congruence. Assume that $ab\sim ac$, but $b\nsim c$. We have to show that $a\sim e$. Let $J=\{d\sim' d'|d\sim d'\text{ or }d=b\text{ and }d'=c\}$. By the maximality of $\sim$, we have $1\sim_J e$. Thus there is a sequence 
 \[
  1 \ \= \ \sum c_{1,k} \ \sim^{J}_\N \quad\dotsb\quad \sim^{J}_\N  \ \sum d_{n,k} \ \= \ e.
 \]
 If $d\sim^J d'$, then $ad\sim ad'$ since $ab\sim ac$. Multiplying the above sequence by $a$ yields
 \[
  a \ \= \ \sum ac_{1,k} \ \sim^\fm_\N \quad\dotsb\quad \sim^\fm_\N  \ \sum ad_{n,k} \ \= \ ae \ \= e
 \]
 and therefore $a\sim e$, which was to be shown. Thus $\sim$ is a prime congruence.
\end{proof}

%%%%%%%%%%%%%%%%%%%%%%%%%%%%%%%%%%%%%%%%%%%%%%%%%%%%%%%%%%%%%%%%%%%%%%%%%%%%%%%%%%%%%%%%%%%%%%%%%%%%%%%%%%%%%%%%%%%%%%%%%%%%%%%%%%%%%%%%%%%%%%%%%%%%%%%%%%%%%%%%%%%%%%%%%%%%%%%%%%%%%%%%%%%%%%%%%%%%%%%%%%%%%%%%%%%%%%%%%%%%

\section{Blue schemes}
\label{section:blue_schemes}

%%%%%%%%%%%%%%%%%%%%%%%%%%%%%%%%%%%%%%%%%%%%%%%%%%%%%%%%%%%%%%%%%%%%%%%%%%%%%%%%%%%%%%%%%%%%%%%%%%%%%%%%%%%%%%%%%%%%%%%%%%%%%%%%%%%%%%%%%%%%%%%%%%%%%%%%%%%%%%%%%%%%%%%%%%%%%%%%%%%%%%%%%%%%%%%%%%%%%%%%%%%%%%%%%%%%%%%%%%%%

\subsection{Definition}
\label{section:definition_of_blue_schemes}

As in classical scheme theory, we will define the spectrum of a blueprint to be a topological space together with a structure sheaf, which is a sheaf in $\bp$ in this case. A blue scheme is a space that is locally isomorphic to spectra of blueprints. We begin with defining the category of blueprinted spaces, in which we embed the category of spectra of monoids and the category of blue schemes.

\begin{pg}
 A \emph{blueprinted space} is a topological space $X$ together with a sheaf $\cO_X$ in $\bp$. A \emph{morphism of blueprinted spaces} is a continuous map $\varphi:X\to Y$ together with a morphism $\varphi^\#:\cO_Y\to\cO_X$ of sheaves. 

 By Proposition \ref{prop:colimits_of_directed_diagrams_in_bp}, colimits of directed systems of blueprints exist, and we can define the \emph{stalk $\cO_{X,x}$ of $\cO_X$ at $x\in X$} as the colimit of all open subsets $U$ of $X$ that contain $x$ together with the inclusion maps. A morphism $\varphi:X\to Y$ of blueprinted spaces induces blueprint morphisms $\varphi^\#_x: \cO_{Y,\varphi(x)}\to\cO_{X,x}$ between stalks for every $x\in X$.

 A blueprinted space $X$ is \emph{locally blueprinted}, if all stalks are \emph{local blueprints}, i.e.\ if $\cO_{X,x}$ has a unique maximal ideal $\fm_x$ for every $x\in X$. A morphism $\varphi:X\to Y$ of locally blueprinted spaces is \emph{local} if for every $x\in X$ and $y=\varphi(x)$, the morphism $\varphi^\#_x: \cO_{Y,y}\to\cO_{X,x}$ between stalks maps the maximal ideal $\fm_y$ of $\cO_{Y,y}$ to the maximal ideal $\fm_x$ of $\cO_{X,x}$.
 
 For a prime ideal $\fp$ of a blueprint $B$, we denote by $B_\fp$ the localization of $B$ at $S=B-\fp$.
\end{pg}

\begin{df}\label{def:spectrum_of_a_blueprint}
 The \emph{spectrum $\Spec B$ of a blueprint $B$} is the blueprinted space whose underlying set is $\Spec B=\{\fp\subset B\text{ prime ideal}\}$. The topology of $\Spec B$ is generated by the sets $U_h=\{\fp\in\Spec B|h\notin \fp\}$ where $h$ ranges through $B$. Note that $U_0=\emptyset$ and $U_1=\Spec B$. Since $U_h\cap U_g=U_{gh}$, the open sets $U_h$ form a basis of the topology. The \emph{structure sheaf $\cO_X$ of $X=\Spec B$} is defined by
 \[
  \cO_X(U) \ = \ \left\{\ s:U\to\coprod_{\fp\in U}B_\fp \ \left| \ \begin{array}{c}\text{For all }\fp\in U,\ s(\fp)\in B_\fp\text{ and there are}\\ 
                                                                                    h\in (B-\fp),\ a\in B\text{ and }n\in\N\text{ such that}\\
                                                                                    U_h\subset U\text{ and for all }\fq\in U_h,\ s(\fq)=\frac{a}{h^n}\text{ in }B_\fq.
                                                                   \end{array}\ \right.\right\}
 \]
 for any open subset $U$ of $B$ where we apply the usual convention that $\cO_X(\emptyset)$ is the terminal object $\bpgenquot{\{0\}}{0\=\zero}$ of $\bp$.
\end{df}

\begin{pg}
 In other words, $\cO_X(U)$ is the set of locally represented sections on $U$. The sheaf $\cO_X$ is indeed a sheaf in $\bp$ since $\cO_X(U)$ caries the following structure for all open subsets $U$ of $X$: for two sections $s,t\in \cO_X(U)$, the section $st$ sends $\fp\in U$ to $s(\fp)\cdot t(\fp)$ in $B_\fp$, and we have $\sum s_i\=\sum t_j$ in $\cO_X(U)$ if and only if $\sum s_i(\fp)\=\sum t_j(\fp)$ in $B_\fp$ for all $\fp\in\cO_X(U)$.

 It is clear that $\cO_{X,\fp}=B_\fp$ (the proof is completely analogous to the case of usual schemes, cf.\ \cite[Prop.\ 2.2 (a)]{Hartshorne}). Since $B_\fp$ is a local blueprint with maximal ideal $\fp B_\fp$, the spectrum $\Spec B$ of $B$ is a locally blueprinted space. 
\end{pg}

\begin{df}
 An \emph{affine blue scheme} is a locally blueprinted space that is isomorphic to the spectrum of a blueprint. A \emph{blue scheme} is a locally blueprinted space that has a covering by affine blue schemes. An open subset that is isomorphic to an affine blue scheme is called an \emph{affine open subscheme}. A \emph{morphism of (affine) blue schemes} is a local morphism of locally blueprinted spaces. We denote the \emph{category of blue schemes} by $\BSch$.
\end{df}

\begin{pg}
 We need more insight into the structure of blue schemes to prove that the affine open subschemes of a blue scheme form a basis of the topology and that a morphism of blue schemes is locally described by morphisms of blueprints. We postpone the proof of these facts to Section \ref{section:affine_open_subschemes}.

 Note that as in the case of usual schemes, a subset $V$ of $\Spec B$ is closed if and only if there is an ideal $I$ of $B$ such that $V$ equals $V(I)=\{\fp\in\Spec B|I\subset \fp\}$ (cf.\ \cite[p.\ 70]{Hartshorne}).
 In particular, the complement of $U_h$ is $V(\igen{h})$. If we define the \emph{radical of $I$} as the ideal 
 \[
  \Rad(I) \ = \ \{ \ a\in B \ | \ \exists n\in \N\text{ such that }a^n\in I \ \},
 \]
 then $V(I)=V(\Rad(I))$.
\end{pg}

\begin{lemma}
 For all $h\in B$, the subspace $U_h$ of $\Spec B$ is compact.
\end{lemma}

\begin{proof}
 Let $U_h=\bigcup_{i\in I} U_i$ be a covering. Since a basis of the topology is given by subsets of the form $U_g$, we may assume that $U_i=U_{h_i}$ for certain elements $h_i\in B$. By taking complements, the condition $U_h=\bigcup_{i\in I} U_i$ becomes equivalent to $V(\igen{\{h_i\}_{i\in I}})=\bigcap_{i\in I}V(\igen{h_i})\subset V(\igen{h})$. This, in turn, means nothing else than $h\in\Rad(\igen{\{h_i\}_{i\in I}})$ resp.\ $h^n\in\igen{\{h_i\}_{i\in I}}$ for some $n\in \N$. By the definition of the ideal generated by $\{h_i\}_{i\in I}$ (cf.\ paragraph \ref{pg:generated_ideal_and_congruence}), this means that we have a sequence
 \[
  h^n \ \= \ \sum c_{1,k} \ \sim^J_\N \quad\dotsb\quad \sim^J_\N \ \sum d_{n,k} \ \= \ b
 \]
 where $J=\{h_i\}_{i\in I}B$, $b\in J$ and $c_{i,k}, d_{i,k}\in B$. Since there are only finitely elements involved in this sequence, there is a finite subset $\{h_i\}_{i\in I_0}$ of $\{h_i\}_{i\in I}$ such that 
 \[
  h^n \ \= \ \sum c_{1,k} \ \sim^{J_0}_\N \quad\dotsb\quad \sim^{J_0}_\N \ \sum d_{n,k} \ \= \ b
 \]
 with $J_0=\{h_i\}_{i\in I_0}B$ and $b\in J_0$. This means that $h^n\in \igen{\{h_i\}_{i\in I_0}}$ and $U_h=\bigcup_{i\in I_0}U_{h_i}$. Thus $U_h$ is compact.
\end{proof}

\begin{pg}
 A morphism $f:B\to C$ of blueprints induces a morphism $f^\ast:\Spec C\to \Spec B$ of locally blue\-printed spaces by taking the inverse image $f^\ast(\fq)=f^{-1}(\fq)$ of prime ideals $\fq$ of $C$. This map is continuous and induces a morphism of sheaves (cf.\ \cite[Prop.\ 2.3]{Hartshorne} for the analogous case of usual schemes). In particular, if $\fq\subset C$ is a prime ideal and $\fp=f^{-1}(\fq)$ is the inverse image of $\fq$, then $f_\fq:B_\fp\to C_\fq$ is a local morphism, i.e.\ $f(\fp B_\fp)\subset \fq C_\fq$. This defines a contravariant functor 
 \[
  \Spec: \ \bp \ \longrightarrow \ \left\{\begin{array}{c}
                                           \text{locally blueprinted spaces}\\\text{and local morphisms}
                                          \end{array} \right\}.
 \]
 This functor is, in contrast to usual scheme theory, not full as the following example shows. More precisely, we see that the functor of global sections is in general not a left-inverse of $\Spec$, and non-isomorphic blueprints can give rise to isomorphic spectra. We will investigate the precise connection between $\Spec$ and the functor of global sections in the next section.
\end{pg}

\begin{ex}\label{ex:spec_and_global_sections}
 Let $B$ be the blueprint associated to the multiplicative subset $A$ of $\Z[S,T]$ that is generated by $\{T,1-T, TS, (1-T)S\}$ (cf.\ Section \ref{section:subsets_of_semirings}). We abbreviate: $h_1=T$, $h_2=1-T$, $a_1=TS$ and $a_2=(1-T)S$. Then the underlying monoid $A$ of $B$ is $\F_1[h_1,h_2,a_1,a_2]/\sim$ where $\F_1=\{0,1\}$ and $\sim$ is the monoid congruence generated by $a_1h_2=a_2h_1$. The pre-addition of $B$ is generated by $h_1+h_2\=1$. 

 To determine the prime ideals of $B$, we note that a prime ideal of $B$ is in particular a prime ideal of the monoid $A$ (with a zero $0$) and a prime ideal of $A$ lifts to a prime ideal of $\F_1[h_1,h_2,a_1,a_2]$. Since prime ideals of $\F_1[h_1,h_2,a_1,a_2]$ are all of the form $\gen{J}$ with $J\subset\{h_1,h_2,a_1,a_2\}$, every prime ideal of $B$ must also be generated by a subset of $\{h_1,h_2,a_1,a_2\}$. One easily verifies case by case that the set of prime ideals of $B$ is $\Spec B=\{(0),\fp_1,\fp_2,\fp_a\}$ where $\fp_1=(h_1,a_1)$, $\fp_2=(h_2,a_2)$ and $\fp_a=(a_1,a_2)$.

 We construct the following global section. Consider the open subsets $U_{h_1}=\{(0),\fp_a,\fp_2\}$ and $U_{h_2}=\{(0),\fp_a,\fp_1\}$, which cover $X=\Spec B$. We define the section $s:X\to\coprod_{\fp\in X} B_\fp$ by
 $s(\fp)=\frac{a_1}{h_1}$ if $\fp\in U_{h_1}$ and $s(\fp)=\frac{a_2}{h_2}$ if $\fp\in U_{h_2}$. Note that in both $B_{(0)}$ and $B_{\fp_a}$, we have $\frac{a_1}{h_1}=\frac{a_2}{h_2}$ since $a_1h_2=a_2h_1$ in $B$; thus $s$ is indeed an element of $\cO_X(X)$.

 As we will verify in the following section, $\Gamma B=\cO_X(X)$ caries the natural structure of a blueprint and the association $a\to (s_a:\fp\to b)$ defines an inclusion $B\to\Gamma B$. In $\Gamma B$, we have the section $s$ as constructed above. It satisfies $s\cdot s_{h_1}=s_{a_1}$ and $s\cdot s_{h_2}=s_{a_2}$. But there is no $b\in B$ such that $bh_1=a_1$ and $bh_2=a_2$. Thus $B\to\Gamma B$ cannot be an isomorphism.

 This shows that the global section functor is not a left-inverse of $\Spec$. Moreover, we will see in the following section that the morphism $B\to\Gamma B$ induces an isomorphism $\Spec \overline B\stackrel\sim\to\Spec B$. Thus $\Spec$ is not full.
\end{ex}

%%%%%%%%%%%%%%%%%%%%%%%%%%%%%%%%%%%%%%%%%%%%%%%%%%%%%%%%%%%%%%%%%%%%%%%%%%%%%%%%%%%%%%%%%%%%%%%%%%%%%%%%%%%%%%%%%%%%%%%%%%%%%%%%%%%%%%%%%%%%%%%%%%%%%%%%%%%%%%%%%%%%%%%%%%%%%%%%%%%%%%%%%%%%%%%%%%%%%%%%%%%%%%%%%%%%%%%%%%%%

\subsection{Global sections and globalization}
\label{section:global_sections}

As demonstrated in Example \ref{ex:spec_and_global_sections}, the global sections $\Gamma B=\Gamma(X,\cO_X)$ of $X=\Spec B$ are in general not equal to $B$. We will show in this section that $\Gamma$ is an idempotent endofunctor on $\bp$, i.e.\ $\Gamma\Gamma B=\Gamma B$. In particular, $\Spec$ restricts to a fully faithful functor from the essential image of $\Gamma$ to the category of locally blueprinted spaces with local morphisms, and the global section functor is its left-inverse.

\begin{pg}
 Let $B$ be a blueprint and $h\in B$. We define a map $\sigma_h:B[h^{-1}]\to \cO_X(U_h)$ as follows. Since for all $\fp\in U_h$, $h$ is invertible in $B_\fp$, we have canonical maps $B[h^{-1}]\to B_\fp$. An element of $B[h^{-1}]$ is of the form $\frac{a}{h^n}$ for some $a\in B$ and $n\in \N$, and we denote the image of $\frac{a}{h^n}$ in $B_\fp$ by the same symbol $\frac{a}{h^n}$. For an element $b=\frac{a}{h^n}\in B[h^{-1}]$ we define the section $\sigma_h(b)=s_{b}:U_h\to \coprod_{\fp\in U_h}$ as $s_b(\fp)=\frac{a}{h^n}\in B_\fp$. 
\end{pg}

\begin{lemma}\label{lemma:blpr_injects_into_sections}
 Let $B$ be a proper blueprint. Then the map $\sigma_h:B[h^{-1}]\to \cO_X(U_h)$ is an injective morphism of blueprints for every $h\in B$.
\end{lemma}

\begin{proof}
 The proof is analogous to the case of usual schemes (cf.\ \cite[Prop\ 2.2 (b)]{Hartshorne}).
\end{proof}

\begin{pg}
 If $B$ is not proper, then $\sigma_h$ is not injective in general: for instance, the spectrum of a blueprint $B=\bpquot A\cR$ with $\cR=\N[A]\times\N[A]$ is the empty scheme whose global sections are the terminal blueprint $\bpgenquot{\{0\}}{0\=\zero}$, independently from $A$. If $A$ contains more than one element, then $\sigma_1$ is not injective. We saw also in Example \ref{ex:spec_and_global_sections} that $\sigma_1$ is in general not surjective. 

 The \emph{functor $\Gamma(\blanc,\cO)$ of global sections} associates to a blue scheme $X$ its global sections $\Gamma(X,\cO_X)=\cO_X(X)$ and to a morphism $\varphi:X\to Y$ the morphism $\varphi^\#(Y): \cO_Y(Y)\to\cO_X(X)$ between global sections. We denote the endofunctor on $\bp$ given as the composition of $\Spec$ with $\Gamma(\blanc,\cO)$ by $\Gamma$. A blueprint $B$ is called \emph{global} if it contained in the essential image of $\Gamma$, i.e.\ if there is a blueprint $C$ such that $B\simeq\Gamma C$. We call the morphism $\sigma=\sigma_1:B\to\Gamma B$ the \emph{globalization of $B$}.

 The most prominent examples of global blueprints are rings. Another class of global blue\-prints are local blueprints $B$: indeed, the only open subset of $\Spec B$ that contains the maximal ideal of $B$ is $\Spec B$ itself; thus $B=\Gamma B$. Since monoids (with a zero) are local blueprints (see \cite[Section 1.1]{LL09b} resp.\ \cite[Section 2.1.1]{CLS10}), they are global blueprints. 
\end{pg}

\begin{thm}\label{thm:global_sections_are_global_blueprints}
 Let $B$ be a blueprint and $\sigma:B\to \Gamma B$ its globalization. Then the induced morphism of spectra $\sigma^\ast:\Spec\Gamma B {\to}\Spec B$ is an isomorphism. 
\end{thm}

Before we proceed to prove the theorem, we state the following immediate consequences.

\begin{cor}\label{cor:globalization_is_global}
 Every morphism from a blueprint $B$ into a global blueprint $C$ factors uniquely through the globalization $\sigma:B\to\Gamma B$. The blueprint $B$ is global if and only if $\sigma:B\to \Gamma B$ is an isomorphism. In particular, $\Gamma^2=\Gamma$ is an endofunctor on $\bp$. \qed
\end{cor}

Let $\bp_\gl$ be the full subcategory of $\bp$ whose objects are global blueprints and let $\bpspaces$ be the category of locally blueprinted spaces together with local morphisms.

\begin{cor}\label{cor:spec_is_fully_faithful_on_global_blueprints}
 The functor $\Spec:\bp\to\bpspaces$ restricts to a fully faithful embedding $\Spec:\bp_\gl\to\bpspaces$, and $\Gamma(\blanc,\cO):\bpspaces\to\bp_\gl$ is a left-inverse of $\Spec$.
\end{cor}

\begin{proof}
 If we show that the global section functor is a left-inverse of $\Spec$ restricted to $\bp_\gl$, then it is clear that $\Spec$ is fully faithful. On the level of objects, this follows from Theorem \ref{thm:global_sections_are_global_blueprints}. On the level of morphisms, it is the same argument as for usual schemes (cf.\ \cite[Prop.\ 2.3]{Hartshorne}).
\end{proof}

 The rest of this section is dedicated to the proof of Theorem \ref{thm:global_sections_are_global_blueprints}. The following technical statement is true for the same reasons as in usual scheme theory (cf.\ \cite[Prop\ 2.2 (b)]{Hartshorne}).

\begin{lemma}\label{lemma:local_representation_of_global_sections}
 Let $h\in B$, $X=\Spec B$ and $s\in\cO_X(U_h)$. Then there are finitely many elements $h_1,\dotsc,h_n, a_1,\dotsc, a_n\in B$ that satisfy the following properties:
 \begin{itemize}
  \item $U_h=\bigcup_{i=1}^n U_{h_i}$;
  \item $s(\fp)=\frac{a_i}{h_i}$ in $B_\fp$ for all $\fp\in U_{h_i}$ and all $i=1,\dotsc,n$;
  \item $a_ih_j=a_jh_i$ for all $i,j=1,\dotsc,n$. \hfill\qed
 \end{itemize}
\end{lemma}

Let $X=\Spec B$. Recall that $\sigma: B\to\Gamma B$ maps $a\in B$ to the section $s_a:X\to\coprod_{\fp\in X}B_\fp$ with $\sigma(\fp)=\frac a1$ in $B_\fp$. The morphism $\sigma^\ast:\Spec \Gamma B\to \Spec B$ maps a prime ideal $\fq$ of $\Gamma B=\cO_X(X)$ to $\fp=\sigma^{-1}(\fq)$. We will prove in the following that the map 
\[
 \begin{array}{cccl}
  \sigma_\ast: & \Spec B & \longrightarrow & \Spec\Gamma B \\
               &  \fp    & \longmapsto     & \sigma_\ast(\fp) \ = \ \{ \ s\in\Gamma B \ | \ s(\fp)\in\fp B_{\fp} \ \}
 \end{array}
\]
 is an inverse of $\sigma^\ast$. It is easily verified that $\sigma_\ast(\fp)$ is indeed a prime ideal of $\Gamma B$. Clearly, we have $\sigma^{-1}(\sigma_\ast(\fp))=\fp$ for every prime ideal $\fp$ of $B$.

\begin{lemma}\label{lemma:sigma*-is-a-bijection}
 The maps $\sigma^\ast$ and $\sigma_\ast$ are mutually inverse bijections.
\end{lemma}

\begin{proof}
 Since $\sigma^{-1}(\sigma_\ast(\fp))=\fp$, the map $\sigma^\ast$ is surjective. We will show that $\fq=\sigma_\ast(\fp)$ is the only prime ideal of $\Gamma B$ that satisfies $\sigma^{-1}(\fq)=\fp$. This implies that $\sigma^\ast$ is injective and consequently bijective. The equality $\sigma^{-1}(\sigma_\ast(\fp))=\fp$ proves in turn that $\sigma_\ast$ is the inverse of $\sigma^\ast$.

 Let $s\in \Gamma B$, i.e.\ there are $h_1,\dotsc,h_n, a_1,\dotsc, a_n\in B$ such that $X=\bigcup_{i=1}^n U_{h_i}$ and such that $s(\fp')=\frac{a_i}{h_i}$ in $B_{\fp'}$ and $a_ih_j=a_jh_i$ for all $\fp'\in U_{h_i}$ and for all $i,j=1,\dotsc,n$ (see Lemma \ref{lemma:local_representation_of_global_sections}). We may assume that $\fp\in U_{h_1}$, i.e.\ $h_1\notin\fp$. Then $(s_{h_1} s)(\fp')=h_1\frac{a_i}{h_i}$ in $B(\fp')$ for all $\fp'\in U_{h_i}$. But $h_1\frac{a_i}{h_i}=a_1$ since $a_1h_i=a_ih_1$, and thus $s_{h_1} s=s_{a_1}=\sigma(a_1)$.

 If $\fq'$ is a prime ideal of $\Gamma B$ with $\sigma^{-1}(\fq')=\fp$, then $s\in\fq'$ if and only if $s_{h_1} s\in\fq'$ since $h_1\notin \fp$ by assumption and consequently $s_{h_1}\notin\fq'$. Since $\sigma^{-1}(\fq')=\fp$, the condition $s_{h_1} s\in\fq'$ is in turn equivalent to $a_1=s_{a_1}(\fp)=(s_{h_1}s)(\fp)\in\fp B_\fp$, or, again since $h_1\notin \fp$, to $\frac{a_1}{h_1}=s(\fp)\in\fp B_\fp$. This shows that $s\in\fq'$ if and only if $s\in\sigma_\ast(\fp)$. Therefore $\fq=\sigma_\ast(\fp)$ is the only prime ideal of $\Gamma B$ that satisfies $\sigma^{-1}(\fq)=\fp$, and the lemma is proven.
\end{proof}

\begin{lemma}
 Let $X=\Spec B$, $Y=\Spec\Gamma B$ and $\varphi=\sigma^\ast: Y\to X$.
 \begin{enumerate}
  \item If $h\in B$, then $\varphi^{-1}(U_h)=U_{s_h}$.
  \item For all $s\in \Gamma B$, there are $h_1,\dotsc,h_n$ such that $\varphi(U_s)=\bigcup_{i=1}^n U_{h_i}$.
 \end{enumerate}
 Consequently, $\varphi$ is a homeomorphism.
\end{lemma}

\begin{proof}
 We know from Lemma \ref{lemma:sigma*-is-a-bijection} that $\varphi$ is a continuous bijection. Part \eqref{part2} implies that $\varphi$ is an open map. Thus it is clear that $\varphi$ is homeomorphism once \eqref{part2} is proven. Part \eqref{part1} follows from
 \[
  \varphi^{-1}(U_h) \ = \ \{ \ \fq\in Y \ | \ h\notin\sigma^{-1}(\fq) \ \} \ = \ \{ \ \fq\in Y \ | \ s_h\notin\fq \ \} \ = \ U_{s_h}.
 \]
 
 We prove part \eqref{part2}. Let $s\in\Gamma B$, i.e.\ there are $h_1,\dotsc,h_n, a_1,\dotsc, a_n\in B$ such that $X=\bigcup_{i=1}^n U_{h_i}$ and such that $s(\fp)=\frac{a_i}{h_i}$ in $B_{\fp}$ and $a_ih_j=a_jh_i$ for all $\fp\in U_{h_i}$ and for all $i,j=1,\dotsc,n$ (see Lemma \ref{lemma:local_representation_of_global_sections}). Then
 \[
  U_s \ = \ \{ \ \fq\in Y \ | \ s\notin \fq \ \} \ = \ \{ \ \fq\in Y \ | \ s(\fp)\notin\fp B_\fp\text{ where }\fp=\varphi(\fq) \ \}.
 \]
 For $i=1,\dotsc,n$, we have thus $\fq\subset U_{s_{h_i}}\cap U_s$ if and only if $h_i\notin\fp$ and $\frac{a_i}{h_i}=s(\fp)\notin\fp B_\fp$ where $\fp=\varphi(\fq)$, or, equivalently, $a_ih_i\notin\fp$. Therefore $\varphi(U_{s_{h_i}}\cap U_s)=U_{a_ih_i}$ and $\varphi(U_s)=\bigcup_{i=1}^nU_{a_ih_i}$.
\end{proof}

\begin{lemma}\label{lemma:isomorphic_stalks_for_sigma*}
 Let $X=\Spec B$, $Y=\Spec\Gamma B$ and $\varphi=\sigma^\ast: Y\to X$. Let $\fq\in Y$ and $\fp=\varphi(\fq)$. Then the induced morphism of stalks $\varphi_\fq: \cO_{X,\fp}\to\cO_{Y,\fq}$ is an isomorphism.
\end{lemma}

\begin{proof}
 Note that $\cO_{X,\fp}\simeq B_\fp$ and $\cO_{Y,\fq}\simeq(\Gamma B)_\fq$. With these identifications, $\varphi_\fq: \cO_{X,\fp}\to\cO_{Y,\fq}$ is nothing else than the localization $\sigma_\fq:B_\fp\to (\Gamma B)_\fq$ of $\sigma$ that maps ${a}/{h}$ to ${s_a}/{s_h}$. 

 We show that $\varphi_\fq$ is injective. Assume ${s_a}/{s_h}={s_{a'}}/{s_{h'}}$ in $\cO_{Y,\fq}$. Then there is an open neighbourhood $U$ of $\fq$ such that ${s_a}/{s_h}={s_{a'}}/{s_{h'}}$ in $\cO_Y(U)$. Applying $\sigma^{-1}$ yields that $a/h=a'/h'$ in $\cO_X(\varphi(U))$, and thus $a/h=a'/h'$ in $\cO_{X,\fp}$.

 We show that $\varphi_\fq$ is surjective. Let $s\in\cO_{Y,\fq}$. Then there is an open neighbourhood $U$ of $\fq$ such that $s\in \cO_Y(U)$ and $s(\fq')=r(\fp')/t(\fp')$ for all $\fq'\in U$ where $\fp'=\varphi(\fq')$,$r\in\Gamma B$ and $t\in(\Gamma B-\fq)$. The sections $r$ and $t$ are in turn locally represented by elements of $B$, i.e.\ there are $a,b\in B$ and $g,h\in(B-\fp)$ such that $r(\fp')=a/h$ and $t(\fp')=b/g$ for all $\fp$ in a suitably small neighbourhood $V$ of $\fp$ with $\varphi^{-1}(V)\subset U$. This means that $s=(s_a/s_h)/(s_b/s_g)=s_{ag}/s_{bh}$ in $\varphi^{-1}(V)$. Thus $s=\varphi_\fq(ag/bh)$.
\end{proof}

We have shown that $\sigma^\ast:\Spec \Gamma B\to\Spec B$ is a homeomorphism that induces isomorphisms on stalks. Therefore $\sigma^\ast$ is an isomorphism of blue schemes. This completes the proof of Theorem \ref{thm:global_sections_are_global_blueprints}.\qed

%%%%%%%%%%%%%%%%%%%%%%%%%%%%%%%%%%%%%%%%%%%%%%%%%%%%%%%%%%%%%%%%%%%%%%%%%%%%%%%%%%%%%%%%%%%%%%%%%%%%%%%%%%%%%%%%%%%%%%%%%%%%%%%%%%%%%%%%%%%%%%%%%%%%%%%%%%%%%%%%%%%%%%%%%%%%%%%%%%%%%%%%%%%%%%%%%%%%%%%%%%%%%%%%%%%%%%%%%%%%

\subsection{Affine open subschemes}
\label{section:affine_open_subschemes}

Recall that an affine open subscheme of a blue scheme is an open subset that is isomorphic to the spectrum of a blueprint. In this section, we show that the affine open subschemes of a blue scheme form a basis for its topology. Consequently, we will see that every morphism of blue schemes is locally algebraic, i.e.\ induced by blueprint morphisms. 

Let $B$ be a blueprint and $X=\Spec B$. For a multiplicative subset $S$ of $B$, we denote the canonical morphism that sends $a\in B$ to $\frac a1\in S^{-1}B$ by $\iota:B\to S^{-1}B$. The induced morphism $\iota^\ast: \Spec S^{-1}B\to X$ of blue schemes maps a prime ideal $\fq$ of $S^{-1}B$ to the prime ideal $\fp=\iota^{-1}(\fq)$ of $B$, which is a prime ideal that has empty intersection with $S$. Thus the image of $\iota^\ast$ is contained in $U_S=\{\fp\in X|\fp\cap S=\emptyset\}$. Indeed, $U_S$ equals the image of $\iota^\ast$ as the following lemma shows.

\begin{lemma}\label{lemma:prime_ideals_of_localizations}
 Let $S$ be a multiplicative subset of $B$ and $\iota:B\to S^{-1}B$ the canonical morphism. Then the induced map $\iota^\ast:\Spec S^{-1}B\to U_S$ is a bijection.
\end{lemma}

\begin{proof}
 Let $Y=\Spec S^{-1}B$ and $U=U_S$. We will show that the map $\iota_\ast: U\to Y$ that maps $\fp$ to $\iota_\ast(\fp)=\fp S^{-1}B$ is an inverse of $\iota^\ast$. We make regular use of the fact that for a prime ideal $\fp'$ of a blueprint $B'$ and a unit $s\in B^\times$, an element $a\in B$ is in $\fp'$ if and only if $sa$ is in $\fp'$.
 
 We show that $\iota_\ast$ is well-defined, i.e.\ that $\fp S^{-1}B$ is a prime ideal of $S^{-1}B$. Let $\fp$ be a prime ideal of $B$ such that $\fp\cap S=\emptyset$. We verify the axioms of a prime ideal for $\fq=\fp S^{-1}B$. If $a/s\in\fq$ and $b/t\in S^{-1}B$, then $a\in\fp$ and thus $(a/s)\cdot(b/t)=ab/st\in\fp S^{-1}B=\fq$. If $S^{-1}$ has a zero $e/s\=\zero$, then $e\=\zero$ in $B$ and thus $e\in\fp$. Consequently $e/s\in\fq$. Given a sequence
 \[
  \frac a{s_a} \ = \ \sum \frac{c_{1,k}}{s_{c_{1,k}}} \ \sim^\fq_\N \quad \dotsb \quad \sim ^\fq_\N \sum \frac{d_{n,k}}{s_{d_{n,k}}} \ = \ \frac b{s_b}
 \]
 in $S^{-1}B$ where $b/s_b\in\fq$, then we can multiply the sequence with a common denominator $s\in S$ and obtain the sequence
 \[
  a{s^a} \ = \ \sum {c_{1,k}}{s^{c_{1,k}}} \ \sim^\fp_\N \quad \dotsb \quad \sim ^\fp_\N \sum {d_{n,k}}{s^{d_{n,k}}} \ = \ b{s^b}
 \]
 in $B$ where the $s^x=s/s_x$ are elements of $S$ and $bs^b\in \fp$. Consequently, $as^a\in\fp$ and $a/s_a\in\fq$. This shows that $\fq$ is an ideal. It is prime since $(a/s)\cdot(b/t)\in\fq$ if and only if $ab\in\fp$, which implies $a\in\fp$ or $b\in\fp$ and, consequently, $a/s\in\fq$ or $b/t\in\fq$.

 We show that $\iota^{-1}(\iota_\ast(\fp))=\fp$ for every $\fp\in U$. Let $\fq=\iota_\ast(\fp)$, then 
 \[
  \iota^{-1}(\fq) \ = \ \{ \ a\in B \ | \ \frac a1\in\fq \ \} \ = \ \{ \ a\in B \ | \ a\in\fp \ \} \ = \ \fp.
 \]
 We show that $\iota_\ast(\iota^{-1}(\fq))=\fq$ for every $\fq\in Y$. Let $\fp=\iota^{-1}(\fq)$, then
 \[
  \iota_\ast(\fp) \ = \ \{ \ \frac as\in S^{-1}B \ | \ \frac as=\frac{a'}{1}\frac{s}{t}\text{ with }a'\in\fp \ \} \ = \ \{ \ \frac as \in S^{-1}B \ | \ a\in\fp \ \} \ = \ \fq.
 \]
 This proves the lemma.
\end{proof}

Let $X=\Spec B$. Every open subset $U$ of $X$ inherits a blue scheme structure from $X$ by restricting the topology and the structure sheaf of $X$ to $U$, i.e.\ $\cO_U(V)=\cO_X(V)$ for all open subsets $V$ of $U$.

\begin{thm}\label{thm:basis_opens_are_affine}
 Let $X=\Spec B$ and $h\in B$. Then the canonical morphism $\iota: B\to B[h^{-1}]$ induces an isomorphism $\iota^\ast:\Spec B[h^{-1}]\to U_h$. 
\end{thm}

\begin{proof}
 Let $Y=\Spec B[h^{-1}]$. Applying Lemma \ref{lemma:prime_ideals_of_localizations} to $S=\{h^i\}_{i\geq0}$ yields that $\varphi=\iota^\ast$ is a continuous bijection between the points of $Y$ and $U_h$. Given a basis open $U_{a/h^n}$ of $Y$, we first note that $U_{a/h^n}=U_{a/1}$ since $1/h^n$ is invertible in $B[h^{-1}]$. Using that $\varphi:Y\to U_h$ is a bijection, we obtain
 \[
  \varphi(U_{a/1}) \ = \ \{ \ \varphi(\fq)\in U_h \ | \ \fq\in Y\text{ and }\frac a1\notin\fq \ \} \ = \ \{ \ \fp \in U_h \ | \ a\notin\fp \ \} \ = \ U_a
 \]
 which is an open subset of $U_h$. Thus $\varphi$ is an open map and, consequently, a homeomorphism.

 To complete the proof of the theorem, we have to show that $\varphi^\#:\cO_{U_h}\to\cO_Y$ is an isomorphism of sheaves. This follows once we have shown that the morphisms $\varphi_\fq:\cO_{U_h,\fp}\to\cO_{Y,\fq}$ are isomorphisms of stalks for all $\fq\in Y$ and $\fp\varphi(\fq)$. Note that $\cO_{U_h,\fp}\simeq B_\fp$ and $\cO_{Y,\fq}\simeq B[h^{-1}]_\fq$ in a canonical way and that $\varphi_\fq$ is nothing else than the localization $\iota_\fq: B_\fp\to B[h^{-1}]_\fq$ of $\iota: B\to B[h^{-1}]$. Since both $B_\fp$ and $B[h^{-1}]_\fq$ satisfy the universal property that every morphism $g: B\to C$ with $g(B-\fp)\subset C^\times$ factors uniquely through $B\to B_\fp$ resp.\ $B\to B[h^{-1}]_\fq$, the morphism $\iota_\fq: B_\fp\to B[h^{-1}]_\fq$ is an isomorphism.
\end{proof}

This theorem has some immediate consequences.

\begin{cor}
 Let $h\in B$. Then $\cO_X(U_h)\simeq\Gamma B[h^{-1}]$.\qed
\end{cor}

\begin{cor}\label{cor:affine_opens_are_a_basis}
 The affine open subschemes of a blue scheme form a basis for its topology.
\end{cor}

\begin{proof}
 Since the open subsets of the form $U_h$ with $h\in B$ form a basis for the topology of $X=\Spec B$, the corollary is true for affine blue schemes. A basis of a blue scheme is given by the union of bases of an open affine covering of the blue scheme. Therefore, the collection of all affine open subschemes is a basis for its topology.
\end{proof}

From this, we can conclude that morphisms between blue schemes are \emph{locally algebraic}, i.e.\ we have the following statement.

\begin{thm}\label{thm:morphisms_are_locally_algebraic}
 Let $\varphi: Y\to X$ be morphism of blue schemes. Then there are open affine coverings $X=\bigcup_{i\in I}U_i$ and $Y=\bigcup_{i\in I} V_i$ of $X$ resp.\ $Y$ where $U_i\simeq\Spec B_i$ and $V_i\simeq\Spec C_i$ for global blueprints $B_i$ and $C_i$ and morphisms $f_i: B_i\to C_i$ of blueprints such that $\varphi\vert_{V_i}=f_i^\ast:V_i\to U_i$.
\end{thm}

\begin{proof}
 Given an open affine covering $X=\bigcup_{i\in I}U_i$, we can cover each open $\varphi^{-1}(U_i)$ by affine open subschemes $V_{i,j}$ of $Y$ by Corollary \ref{cor:affine_opens_are_a_basis}. If we define $U_{i,j}=U_i$, then $\varphi(V_{i,j})\subset U_{i,j}$ for each $i$ and $j$. After replacing $(i,j)$ by a new index $i$, we can assume thus that we have an open affine covering $X=\bigcup_{i\in I}U_i$ of $X$ and an open affine covering $Y=\bigcup_{i\in I} V_i$ of $Y$ with $\varphi(V_i)\subset U_i$. Let $B_i=\cO_X(U_i)$ and $C_i=\cO_Y(V_i)$, then these are global blueprints and we know from Corollary \ref{cor:spec_is_fully_faithful_on_global_blueprints} that $U_i\simeq\Spec B_i$, that $V_i\simeq \Spec C_i$ and that $f_i=\varphi^\#(U_i): B_i\to C_i$ satisfies that $\varphi\vert_{V_i}=f_i^\ast:V_i\to U_i$.
\end{proof}

%%%%%%%%%%%%%%%%%%%%%%%%%%%%%%%%%%%%%%%%%%%%%%%%%%%%%%%%%%%%%%%%%%%%%%%%%%%%%%%%%%%%%%%%%%%%%%%%%%%%%%%%%%%%%%%%%%%%%%%%%%%%%%%%%%%%%%%%%%%%%%%%%%%%%%%%%%%%%%%%%%%%%%%%%%%%%%%%%%%%%%%%%%%%%%%%%%%%%%%%%%%%%%%%%%%%%%%%%%%%

\subsection{Fibre products}
\label{section:fibre_products}

In this section, we will show that the usual construction of fibre products of schemes is adaptable to blue schemes. We begin with explaining how to glue blue schemes along open subschemes.

\begin{pg}
 An \emph{open subscheme of a blue scheme $X$} is an open subset $U$ of $X$ together with the relative topology and the restriction of structure sheaf $\cO_X$ to $U$. A morphism of blue schemes $\varphi:Y\to X$ is an \emph{open immersion} if $U=\im\varphi$ is an open subscheme of $X$ and if $\varphi:Y\to U$ is an isomorphism of blue schemes.

 Let $X$ be a blue scheme and $\cU$ a \emph{based system of open subschemes}, i.e.\ a family of open subschemes $U_i$ of $X$ together with the inclusion maps $\varphi_{i,j}:U_i\hookrightarrow U_j$, which are open immersions, such that for every $U_i$ and $U_j$, there are $U_{k_1},\dotsc, U_{k_n}$ such that $U_i\cap U_j=U_{k_1}\cup\dotsb\cup U_{k_n}$ (where the intersection and union are taken inside $X$). Then $X$ is the colimit of $\cU$ in $\bpspaces$. If, conversely, $\cU$ is a system of blue schemes and open immersions, then the colimit $X=\colim\cU$ exists in $\bpspaces$ and $\cU$ is a based system of open subschemes for $X$. In particular, the family of all affine open subschemes of $X$ together with the inclusion maps is a based system of open subschemes. 
\end{pg}

\begin{lemma}
 Let $X$ be a blue scheme. Then there is a morphism $\gamma:X\to \Spec\Gamma(X,\cO_X)$ of blue schemes such that every morphism $\varphi:X\to Y$ into an affine blue scheme $Y\simeq\Spec B$ factors uniquely through $\gamma$.
\end{lemma}

\begin{proof}
 By Theorem \ref{thm:morphisms_are_locally_algebraic}, there is a covering $X=\bigcup U_i$ by open affine subschemes $U_i\simeq\Spec B_i$ of $X$ and blueprint morphisms $B\to B_i$ such that $\varphi\vert_{U_i}=f_i^\ast:\Spec B_i\to\Spec B$. This means that we have commutative diagrams
 \[
  \xymatrix@C=4pc@R=2pc{\Gamma(X,\cO_X)\ar[d]_{\res_{X|U_i}}  & B \ar[l]_{\varphi^\#} \ar[dl]^{f_i}\\   B_i}
  \qquad\text{and}\qquad
  \xymatrix@C=4pc@R=2pc{\Spec\Gamma(X,\cO_X)  \ar[r]& \Spec B \\   U_i \ar[u]\ar[ur]_{f_i^\ast} \ar@{}[r]|\subset & X \ar[u]_\varphi}
 \]
 for every $i$. Since $X$ is covered by the open subschemes $U_i$, the morphisms $U_i\to \Spec\Gamma(X,\cO_X)$ glue to a morphism $\gamma:X\to \Spec\Gamma(X,\cO_X)$, which satisfies the claimed properties of the lemma. 
\end{proof}

\begin{cor}
 For every blue scheme $X$ and for every blueprint $B$, we have $\Hom(X,\Spec B)\simeq\Hom(B,\Gamma(X,\cO_X)$.\qed 
\end{cor}

With these facts at hand, we can adopt the usual construction of fibre products to blue schemes (e.g.\ cf.\ \cite[Thm.\ 3.3]{Hartshorne}).

\begin{prop}
 Given two morphisms $X\to Z$ and $Y\to Z$. Then the fibre product $X\times_ZY$ exists in $\BSch$. If $X\simeq\Spec B_1$, $Y\simeq\Spec  B_2$ and $Z\simeq\Spec  B_0$, then $X\times_ZY\simeq\Spec(B_1\otimes_{B_0}B_2)$.\qed
\end{prop}

%%%%%%%%%%%%%%%%%%%%%%%%%%%%%%%%%%%%%%%%%%%%%%%%%%%%%%%%%%%%%%%%%%%%%%%%%%%%%%%%%%%%%%%%%%%%%%%%%%%%%%%%%%%%%%%%%%%%%%%%%%%%%%%%%%%%%%%%%%%%%%%%%%%%%%%%%%%%%%%%%%%%%%%%%%%%%%%%%%%%%%%%%%%%%%%%%%%%%%%%%%%%%%%%%%%%%%%%%%%%

\subsection{Residue fields}
\label{section:residue_fields}

 In this section, we define the residue field of a point of a blue scheme. 

\begin{lemma}
 Let $\fp$ be a prime ideal of $B$ and $S=B-\fp$. Then there is a unique morphism $B\to\kappa(\fp)$ into a blue field $\kappa(\fp)$ such that every morphism $B\to C$ that maps $\fp$ to a zero of $C$ and $S$ to the units of $C$ factors through $B\to\kappa(\fp)$. The blue field $\kappa(\fp)$ is isomorphic to both $B_\fp/\fp B_\fp$ and $(B/\fp)_{(\fp)}$.
\end{lemma}

\begin{proof}
 It is clear that both $B_\fp/\fp B_\fp$ and $(B/\fp)_{(\fp)}$ together with their canonical epimorphisms $B\to B_\fp/\fp B_\fp$ resp.\ $B\to (B/\fp)_{(\fp)}$ satisfy the universal property of $\kappa(\fp)$. Thus they are isomorphic to $\kappa(\fp)$ and witness the existence of $\kappa(\fp)$. It is clear that $B_\fp/\fp B_\fp$ is a blue field.
\end{proof}

We call $\kappa(\fp)$ the \emph{residue field of $B$ at $\fp$}. Let $X$ is a blue scheme and $x$ a point of $X$. Let $\fm_x$ be the maximal ideal of $\cO_{X,x}$. Then we call $\kappa(x)=\cO_{X,x}/\fm_x$ the \emph{residue field of $X$ at $x$}.

%%%%%%%%%%%%%%%%%%%%%%%%%%%%%%%%%%%%%%%%%%%%%%%%%%%%%%%%%%%%%%%%%%%%%%%%%%%%%%%%%%%%%%%%%%%%%%%%%%%%%%%%%%%%%%%%%%%%%%%%%%%%%%%%%%%%%%%%%%%%%%%%%%%%%%%%%%%%%%%%%%%%%%%%%%%%%%%%%%%%%%%%%%%%%%%%%%%%%%%%%%%%%%%%%%%%%%%%%%%%

\subsection{Subcategories of $\BSch$}
\label{section:subcategories_of_bsch}

In this section, we define certain subcategories of $\BSch$ that come from the subcategories $\bp_\canc$, $\bp_\proper$, $\bp_\inv$, $\bp_0$, $\cM$, $\cM_0$, $\SRings$ and $\Rings$ of $\bp$. In particular, we will identify the categories of usual schemes, $\cM$-schemes and $\cM_0$-schemes with subcategories of $\bp$, as explained in more detail in the following two sections.

\begin{pg}\label{pg:c-schemes}
 Given a property $\cE$ of blueprints that is stable under localization, we can define $\cE$-schemes as follows. A blue scheme $X$ is an $\cE$-scheme if it has an open affine cover $X=\bigcup U_i$ where $U_i\simeq\Spec B_i$ and all $B_i$ have property $\cE$. This definition is not independent of the chosen cover since we did not ask $\cE$ to hold for a blueprint if it holds for a localization (cf.\ Remark \ref{rem:spec_f1_times_f1}). But every open affine cover of an $\cE$-scheme has a refinement by spectra of blueprints with property $\cE$. A blue scheme $X$ is an $\cE$-scheme if and only if $X$ has a basis of affine open subschemes that are spectra of blueprints with $\cE$.

 If $\cC$ is a full subcategory of $\bp$ that contains all localizations, then a $\cC$-scheme is a blue scheme that has an open affine cover $X=\bigcup U_i$ where $U_i\simeq\Spec B_i$ and all $B_i$ are in $\cC$. This applies to all subcategories $\bp_\canc$, $\bp_\proper$, $\bp_\inv$, $\bp_0$, $\cM$, $\cM_0$, $\SRings$ resp.\ $\Rings$ of $\bp$. %Before we define the corresponding subcategories of $\BSch$ we explain another mechanism that applies to all these categories, but $\cM$ and $\cM_0$.
\end{pg}
 
\begin{pg}\label{pg:base_extensions_of_blue_schemes}
 Let $\cC$ be a full subcategory of $\bp$ that contains all localizations and let $\iota:\cC\to \bp$ be the inclusion functor. Assume that there is for every blueprint $B$ a blueprint $\cG(B)$ in $\cC$ and a morphism $f:B\to\cG(B)$ such that every morphism from $B$ into a blueprint in $\cC$ factors uniquely through $f$. This defines a functor $\cG:\bp\to\cC$, which is a left-adjoint of $\iota: \cC\to\bp$. 
\end{pg}

\begin{prop}\label{prop:base_extensions_of_blue_schemes}
 Let $\cC$ and $\cG$ be as above. Then we can associate to every blue scheme $X$ a $\cC$-scheme $\cG(X)$ and a morphism $\beta:\cG(X)\to X$ that satisfies the universal property that every morphism from a $\cC$-scheme into $X$ factors uniquely through $\beta$.
\end{prop}

\begin{proof}
 The $\cC$-scheme $\cG(X)$ is constructed as follows. Let $\cU$ be the based system of all affine open subschemes $U_i\simeq\Spec B_i$ of $X$ together with the inclusion maps $\varphi_{i,j}:U_i\to U_j$. The canonical morphisms $f_i: B_i\to\cG(B_i)$ define morphisms $\psi=f_i^\ast:\Spec\cG(B_i)\to\Spec B_i$ of blue schemes. Note that open immersions are stable under base extensions, i.e.\ if $\iota:U\to X$ is an open immersion and $\varphi:Y\to X$ is a morphism of blue schemes, then $U\times_X Y$ is $\varphi^{-1}(\iota(U))$ and the inclusion $\iota':\varphi^{-1}(\iota(U))\to Y$ is an open immersion. Thus the open immersions $\varphi_{i,j}:U_i\to U_j$ define open immersions $\cG(\varphi_{i,j}):\cG(U_i)\to\cG(U_j)$ and we obtain a system $\cG(\cU)$ of $\cC$-schemes and open immersions. Further we have commutative diagrams
 \[
  \xymatrix@R=1pc@C=3pc{\cG(U_i)\ar[dd]_{\cG(\varphi_{i,j})}\ar[r]^{\psi_i}  &  U_i\ar[dd]_{\varphi_{i,j}}\ar[dr]  \\
                                                                             &                                     & X \\
                        \cG(U_j)\ar[r]^{\psi_j}  &  U_j\ar[ur] }
 \]
 for every $i$ and $j$ such that there is an open immersion $\varphi_{i,j}:U_i\to U_j$. This defines a morphism $\beta:\colim \cG(\cU)\to X$ and we define $\cG(X)$ as $\colim \cG(\cU)$, which is a $\cC$-scheme. The claimed universal property of $\beta:\cG(X)\to X$ follows from Theorem \ref{thm:morphisms_are_locally_algebraic}.
\end{proof}

\begin{pg}
 A blue scheme is called \emph{cancellative / proper / with inverses / with a zero} if it has a cover by affine open subschemes that are spectra of blueprints that are cancellative / proper / with inverses / with a zero. 

 We denote the \emph{full subcategory of cancellative blue schemes} by $\BSch_\canc$. We have for every blueprint $B$ a universal morphism $B\to B_\canc$ into a cancellative blueprint (cf.\ paragraph \ref{pg:generated_pre-addition}). Thus we can apply Proposition \ref{prop:base_extensions_of_blue_schemes} to obtain for every blue scheme $X$ a universal morphism $X_\canc\to X$ from a cancellative blue scheme $X_\canc$ into $X$.
 
 We denote the \emph{full subcategory of proper blue schemes} by $\BSch_\proper$. Since we have universal morphisms $B\to B_\proper$ (cf.\ Lemma \ref{lemma_proper}), Proposition \ref{prop:base_extensions_of_blue_schemes} yields for every blue scheme $X$ a universal morphism $X_\proper\to X$ from a proper blue scheme $X_\proper$ into $X$.
 
 We denote the \emph{full subcategory of proper blue schemes with inverses} by $\BSch_\inv$. Since we have universal morphisms $B\to B_\inv$ (cf.\ Lemma \ref{lemma_inv} \eqref{part1}), Proposition \ref{prop:base_extensions_of_blue_schemes} yields for every blue scheme $X$ a universal morphism $X_\inv\to X$ from a proper blue scheme $X_\inv$ with inverses into $X$.
 
 We denote the \emph{full subcategory of proper blue schemes with a zero} by $\BSch_0$. Since we have universal morphisms $B\to B_0$ (cf.\ Lemma \ref{lemma_inv} \eqref{part2}), Proposition \ref{prop:base_extensions_of_blue_schemes} yields for every blue scheme $X$ a universal morphism $X_0\to X$ from a proper blue scheme $X_0$ with a zero into $X$.
\end{pg}

%%%%%%%%%%%%%%%%%%%%%%%%%%%%%%%%%%%%%%%%%%%%%%%%%%%%%%%%%%%%%%%%%%%%%%%%%%%%%%%%%%%%%%%%%%%%%%%%%%%%%%%%%%%%%%%%%%%%%%%%%%%%%%%%%%%%%%%%%%%%%%%%%%%%%%%%%%%%%%%%%%%%%%%%%%%%%%%%%%%%%%%%%%%%%%%%%%%%%%%%%%%%%%%%%%%%%%%%%%%%

\subsection{Monoidal schemes}
\label{section:monoidal_schemes}

  A blue scheme $X$ is called a \emph{monoidal scheme (with a zero)} if it has an open affine cover $X=\bigcup U_i$ where $U_i\simeq\Spec B_i$ for monoids $B_i$ (with a zero). The embedding $\iota_\cM:\cM\to\bp$ of monoids into blueprints extends to an embedding of $\cM$-schemes (cf.\ \cite[Section 1.1]{LL09b} for a definition) into blue schemes whose essential image is the full subcategory of monoidal schemes. Thus there is no confusion if we call monoidal schemes $\cM$-schemes in accordance with paragraph \ref{pg:c-schemes}. We denote the full subcategory of $\cM$-schemes by $\Sch_\cM$.

 Similarly, the category of $\cM_0$-schemes (cf.\ \cite[Section 1.6]{LL09b} for a definition) embeds into $\BSch$, and its essential image is the full subcategory of monoidal schemes with a zero. Therefore, we may call a monoidal scheme with a zero an $\cM_0$-scheme. We denote the full subcategory of $\cM_0$-schemes by $\Sch_{\cM_0}$.

\begin{rem}\label{rem:spec_f1_times_f1}
 Note that in general not every open affine subscheme of an $\cM_0$-scheme with a zero is a spectrum of a monoid with a zero. To see this, consider the following example. The disjoint union $X=\Spec\Fun\amalg\Spec\Fun$ is a monoidal scheme with a zero since $\Fun=\{0,1\}$ is a monoid with a zero. The global sections of $X$ are $\Gamma(X,\cO_X)=\Fun\times\Fun$, which is not a monoid with a zero since it contains the additive relation $(0,1)+(1,0)\=(0,0)+(1,1)$. But $X=\Spec\Fun\times\Fun$ is an affine blue scheme.
\end{rem}

%%%%%%%%%%%%%%%%%%%%%%%%%%%%%%%%%%%%%%%%%%%%%%%%%%%%%%%%%%%%%%%%%%%%%%%%%%%%%%%%%%%%%%%%%%%%%%%%%%%%%%%%%%%%%%%%%%%%%%%%%%%%%%%%%%%%%%%%%%%%%%%%%%%%%%%%%%%%%%%%%%%%%%%%%%%%%%%%%%%%%%%%%%%%%%%%%%%%%%%%%%%%%%%%%%%%%%%%%%%%

\subsection{Grothendieck schemes and semiring schemes}
\label{section:grothendieck_schemes}

 A \emph{Grothendieck scheme} is a blue scheme that has an open affine cover $X=\bigcup U_i$ where $U_i\simeq\Spec B_i$ for rings $B_i$. This is the same as a $\Rings$-scheme after paragraph \ref{pg:c-schemes}. The embedding $\fB:\Rings\to\bp$ of rings into blueprints extends to an embedding $\fB:\Sch\to\BSch$ of schemes (in the usual sense) into blue schemes whose essential image is the subcategory of Grothendieck schemes. Thus we may think of schemes as Grothendieck schemes, and we denote the category of Grothendieck schemes by $\Sch$. Since every blueprint has a universal morphism $B\to B_\Z$ to a ring, every blue scheme $X$ has a universal morphism from a Grothendieck scheme $X_\Z$ to $X$, following Proposition \ref{prop:base_extensions_of_blue_schemes}. In particular, we obtain the functor $(\blanc)_\Z:\BSch\to\Sch$, which is a right-adjoint to $\fB:\Sch\to\BSch$.

 Note that the composition $(\blanc)_\Z\circ\iota_{\cM}:\Sch_\cM\to\Sch$ coincides with the base extension functor $\blanc\otimes_\Fun\Z$ in $\Fun$-geometry (cf.\ \cite[Section 1.1]{LL09b}). The same is true for $(\blanc)_\Z\circ\iota_{\cM_0}:\Sch_{\cM_0}\to\Sch$ (cf.\ \cite[Section 1.6]{LL09b}).

 Similar to the case of Grothendieck schemes, we define a \emph{semiring scheme} as a blue scheme that can be covered by spectra of semirings. Since every blueprint has a universal morphism $B\to B_\N$ to a semiring, every blue scheme $X$ has a universal morphism from a semiring scheme $X_\N$ to $X$. 

 Note that there are familiar objects occurring in literature under a different name. 

 To\"en and Vaqui\'e apply their general machinery of associating a category of schemes to a symmetric monoidal category (with some extra properties, see \cite{ToenVaquie08}) to the category of $\N$-modules. They call the objects in the resulting category $\N$-schemes, and it seems likely that their category of $\N$-schemes is equivalent to the category of semiring schemes. A proof of this is desirable.

 Durov defines generalized schemes in \cite{Durov07}, which is a generalization of scheme theory in a different direction by considering monads on a category. The resulting structures are called generalized rings, and it is possible to consider semirings as generalized rings. Therefore generalized schemes in Durov's sense cover semiring schemes.

 As already mentioned in the introduction, Mikhalkin defines tropical schemes in \cite{Mikhalkin}. I expect that the connection between tropical schemes and semiring schemes will be clarified by a rigorous theory of congruence schemes for blueprints.

%%%%%%%%%%%%%%%%%%%%%%%%%%%%%%%%%%%%%%%%%%%%%%%%%%%%%%%%%%%%%%%%%%%%%%%%%%%%%%%%%%%%%%%%%%%%%%%%%%%%%%%%%%%%%%%%%%%%%%%%%%%%%%%%%%%%%%%%%%%%%%%%%%%%%%%%%%%%%%%%%%%%%%%%%%%%%%%%%%%%%%%%%%%%%%%%%%%%%%%%%%%%%%%%%%%%%%%%%%%%

\subsection{Connes and Consani's $\Fun$-schemes revised}
\label{section:CC-schemes}

 From the viewpoint on blueprints as multiplicative maps $f:A\to R$ from a monoid $A$ to a semiring $R$ (cf.\ Section \ref{section:subsets_of_semirings}), a familiarity with the concept of an $\Fun$-scheme in the sense of Connes and Consani (see \cite{CC09}) is visible. While blueprints are the fusion of a monoid with a semiring and a blue scheme is the geometric object associated to the glued algebraic object, Connes and Consani associate geometric objects to monoids (namely, $\cM_0$-schemes) and rings (namely, Grothendieck schemes) and combine two of those geometric objects to define an $\Fun$-scheme in their sense. Though the ideas are very similar in thought, the resulting theories are very different. The a posteriori combination of both ideas seems to be important for later applications to zeta functions.

 We review Connes and Consani's $\Fun$-schemes, for short, Connes-Consani schemes, within the framework of blue schemes. As usual in algebraic geometry, we denote by $X(k)$ the morphism set $\Hom(\Spec k, X)$ where $X$ is a blue scheme and $k$ a blueprint. Note that if $k$ is a ring, then every morphism $\Spec k\to X$ factors through $X_\Z$; thus $X(k)=X_\Z(k)$. A \emph{Connes-Consani scheme} is a morphism $\varphi:X_\Z\to Y$ where $X$ is an $\cM_0$-scheme and $Y$ a scheme such that the induced map $\varphi(k):X(k)=X_\Z(k)\to Y(k)$ is a bijection for all fields $k$. 

 The importance of Connes-Consani schemes lies in the theory of their zeta functions, which exhibit information about the number of $\F_q$-rational points (cf.\ \cite{CC09, L10, Soule04}). The concept of Connes and Consani will apply to the more general class of \emph{blue Connes-Consani schemes}, i.e.\ morphisms $\varphi:X_\Z\to Y$ of blue schemes where $X$ is an $\cM_0$-scheme and $X$ an arbitrary blue scheme such that the induces map $X(k)\to Y(k)$ is a bijection for all fields $k$. 

 It is an interesting problem how the combinatorial zeta functions of (blue) Connes-Consani schemes connect to the zeta functions of Grothendieck schemes. Is there a definition of a zeta function for blue schemes that reproduces both kinds of zeta functions?

\begin{small}

\end{small}

\end{document}